%% file: main_new_intro.tex
\documentclass{amsart}

\usepackage[margin=28mm]{geometry}
\usepackage[utf8]{inputenc}

\usepackage{amsmath,amssymb,amstext,amsthm} 
\usepackage{pinlabel}
\usepackage{enumitem}
\usepackage{color}
\usepackage{algorithm}
\usepackage{multicol}
\usepackage{soul}
\newlist{steps}{enumerate}{1}
\setlist[steps, 1]{label = Step \arabic*:}

\usepackage[final]{showkeys}
\usepackage[american]{babel}
\usepackage{hyperref}
\usepackage[obeyDraft]{todonotes} 

\usepackage{amsfonts}
\usepackage{amssymb}
\usepackage{amsthm}
\usepackage{amsmath}
\usepackage{graphicx}
\usepackage{url}
\usepackage[all]{xy}
\usepackage{lscape}
\usepackage{mathtools}
\usepackage{tikz}
\usepackage{float}
\usetikzlibrary{shapes.geometric}
\usetikzlibrary{shapes}
\usetikzlibrary{arrows}
\usetikzlibrary{decorations.markings}
\usetikzlibrary{tqft}
\usetikzlibrary{positioning}
\usetikzlibrary{patterns}
\usepackage{dsfont}
\usepackage{tikz-cd}
\usepackage{fancyhdr}

\usepackage{cleveref}

\makeatletter
\renewcommand\subsection{\@startsection{subsection}{2}%
  \z@{-.5\linespacing\@plus-.7\linespacing}{.5\linespacing}%
  {\normalfont\scshape}}

\makeatletter
\newcommand{\bigsymbol}[1]{%
  \DOTSB
  \mathop{
    \mathchoice{\big@symbol\displaystyle\Large{#1}}
               {\big@symbol\textstyle\large{#1}}
               {\big@symbol\scriptstyle\normalsize{#1}}
               {\big@symbol\scriptscriptstyle\small{#1}}%
    }\slimits@
}

\newcommand{\big@symbol}[3]{%
  \vcenter{%
    \sbox\z@{$#1\sum$}%
    \dimen@=0.675\dimexpr\ht\z@+\dp\z@\relax 
    #2
    \hbox{\resizebox{!}{\dimen@}{$\m@th#3$}}%
  }%
  \vphantom{\sum}%
}
\makeatother

\usepackage{amsthm}

\theoremstyle{definition}
\newtheorem{definition}{Definition}[section]
\newtheorem{example}[definition]{Example}

\theoremstyle{remark}
\newtheorem{remark}[definition]{Remark}

\theoremstyle{plain}
\newtheorem{theorem}[definition]{Theorem}
\newtheorem{proposition}[definition]{Proposition}

\newtheorem{corollary}[definition]{Corollary}

\newtheorem{question}[definition]{Question}

\usepackage{amssymb}

\newcommand{\defeq}{\mathrel{\mathop:}=}

\newcommand{\ie}{i.\,e.~}

\newcommand{\eg}{e.\,g.~}

\newcommand{\R}{\mathbb{R}}
\newcommand{\F}{\mathbb{F}}

\newcommand{\T}{\mathcal{T}}

\newcommand{\sigij}{a_{i,j}}

\usepackage{aliascnt}
\usepackage{xspace}
\usepackage{hyperref}
\usepackage{caption}
\usepackage{subcaption}

\usepackage{yfonts}
\usepackage{psfrag}

\usepackage{chngcntr}

\crefname{thm}{theorem}{theorems}

\usepackage{tikz}
\usetikzlibrary{
  cd,
  calc,
  positioning,
  fit,
  arrows,
  decorations.pathreplacing,
  decorations.markings,
  shapes.geometric,
  backgrounds,
  bending
}

\newcommand*{\StrikeThruDistance}{0.15cm}%
\usepackage{tikzsymbols}

\def\pgf@sh@@knotanchor#1#2{%
  \anchor{#2 north west}{%
    \csname pgf@anchor@knot #1@north west\endcsname%
    \pgf@x=#2\pgf@x%
    \pgf@y=#2\pgf@y%
  }%
  \anchor{#2 north east}{%
    \csname pgf@anchor@knot #1@north east\endcsname%
    \pgf@x=#2\pgf@x%
    \pgf@y=#2\pgf@y%
  }%
  \anchor{#2 south west}{%
    \csname pgf@anchor@knot #1@south west\endcsname%
    \pgf@x=#2\pgf@x%
    \pgf@y=#2\pgf@y%
  }%
  \anchor{#2 south east}{%
    \csname pgf@anchor@knot #1@south east\endcsname%
    \pgf@x=#2\pgf@x%
    \pgf@y=#2\pgf@y%
  }%
  \anchor{#2 north}{%
    \csname pgf@anchor@knot #1@north\endcsname%
    \pgf@x=#2\pgf@x%
    \pgf@y=#2\pgf@y%
  }%
  \anchor{#2 east}{%
    \csname pgf@anchor@knot #1@east\endcsname%
    \pgf@x=#2\pgf@x%
    \pgf@y=#2\pgf@y%
  }%
  \anchor{#2 west}{%
    \csname pgf@anchor@knot #1@west\endcsname%
    \pgf@x=#2\pgf@x%
    \pgf@y=#2\pgf@y%
  }%
  \anchor{#2 south}{%
    \csname pgf@anchor@knot #1@south\endcsname%
    \pgf@x=#2\pgf@x%
    \pgf@y=#2\pgf@y%
  }%
}

\makeatletter
\def\pgfaddtoshape#1#2{%
  \begingroup
  \def\pgf@sm@shape@name{#1}%
  \let\anchor\pgf@sh@anchor
  #2%
  \endgroup
}

\def\useanchor#1#2{\csname pgf@anchor@#1@#2\endcsname}

\def\@shiftback#1#2#3#4#5#6{%
    \advance\pgf@x by -#5\relax
    \advance\pgf@y by -#6\relax
}

\pgfaddtoshape{rectangle}{%
  \anchor{west south west}{%
    \pgf@process{\northeast}%
    \pgf@ya=.5\pgf@y%
    \pgf@process{\southwest}%
    \pgf@y=1.5\pgf@y%
    \advance\pgf@y by \pgf@ya%
    \pgf@y=.5\pgf@y%
  }%
  \anchor{west north west}{%
    \pgf@process{\northeast}%
    \pgf@ya=1.5\pgf@y%
    \pgf@process{\southwest}%
    \pgf@y=.5\pgf@y%
    \advance\pgf@y by \pgf@ya%
    \pgf@y=.5\pgf@y%
  }%
  \anchor{east north east}{%
    \pgf@process{\southwest}%
    \pgf@ya=.5\pgf@y%
    \pgf@process{\northeast}%
    \pgf@y=1.5\pgf@y%
    \advance\pgf@y by \pgf@ya%
    \pgf@y=.5\pgf@y%
  }%
  \anchor{east south east}{%
    \pgf@process{\southwest}%
    \pgf@ya=1.5\pgf@y%
    \pgf@process{\northeast}%
    \pgf@y=.5\pgf@y%
    \advance\pgf@y by \pgf@ya%
    \pgf@y=.5\pgf@y%
  }%
  \anchor{north north west}{%
    \pgf@process{\southwest}%
    \pgf@xa=1.5\pgf@x%
    \pgf@process{\northeast}%
    \pgf@x=.5\pgf@x%
    \advance\pgf@x by \pgf@xa%
    \pgf@x=.5\pgf@x%
  }%
  \anchor{north north east}{%
    \pgf@process{\southwest}%
    \pgf@xa=.5\pgf@x%
    \pgf@process{\northeast}%
    \pgf@x=1.5\pgf@x%
    \advance\pgf@x by \pgf@xa%
    \pgf@x=.5\pgf@x%
  }%
  \anchor{south south west}{%
    \pgf@process{\northeast}%
    \pgf@xa=.5\pgf@x%
    \pgf@process{\southwest}%
    \pgf@x=1.5\pgf@x%
    \advance\pgf@x by \pgf@xa%
    \pgf@x=.5\pgf@x%
  }%
  \anchor{south south east}{%
    \pgf@process{\northeast}%
    \pgf@xa=1.5\pgf@x%
    \pgf@process{\southwest}%
    \pgf@x=.5\pgf@x%
    \advance\pgf@x by \pgf@xa%
    \pgf@x=.5\pgf@x%
  }%
  \anchor{width}{%
    \useanchor{rectangle}{west}%
    \pgf@xc=\pgf@x
    \useanchor{rectangle}{east}%
    \advance\pgf@x by -\pgf@xc
    \pgf@y=\z@
    \edef\pgf@temp{\csname pgf@sh@nt@\pgfreferencednodename\endcsname}%
    \expandafter\@shiftback\pgf@temp
  }
  \anchor{height}{%
    \useanchor{rectangle}{south}%
    \pgf@yc=\pgf@y
    \useanchor{rectangle}{north}%
    \advance\pgf@y by -\pgf@yc
    \pgf@x=\z@
    \edef\pgf@temp{\csname pgf@sh@nt@\pgfreferencednodename\endcsname}%
    \expandafter\@shiftback\pgf@temp
  }
  \anchor{size}{%
    \useanchor{rectangle}{south west}%
    \pgf@xc=\pgf@x
    \pgf@yc=\pgf@y
    \useanchor{rectangle}{north east}%
    \advance\pgf@x by -\pgf@xc
    \advance\pgf@y by -\pgf@yc
    \edef\pgf@temp{\csname pgf@sh@nt@\pgfreferencednodename\endcsname}%
    \expandafter\@shiftback\pgf@temp
  }
}
\makeatother

\pgfdeclareshape{knot crossing}
{
  \inheritsavedanchors[from=circle] 
  \inheritanchorborder[from=circle]
  \inheritanchor[from=circle]{north}
  \inheritanchor[from=circle]{north west}
  \inheritanchor[from=circle]{north east}
  \inheritanchor[from=circle]{center}
  \inheritanchor[from=circle]{west}
  \inheritanchor[from=circle]{east}
  \inheritanchor[from=circle]{mid}
  \inheritanchor[from=circle]{mid west}
  \inheritanchor[from=circle]{mid east}
  \inheritanchor[from=circle]{base}
  \inheritanchor[from=circle]{base west}
  \inheritanchor[from=circle]{base east}
  \inheritanchor[from=circle]{south}
  \inheritanchor[from=circle]{south west}
  \inheritanchor[from=circle]{south east}
  \inheritanchorborder[from=circle]
  \pgf@sh@@knotanchor{crossing}{2}
  \pgf@sh@@knotanchor{crossing}{3}
  \pgf@sh@@knotanchor{crossing}{4}
  \pgf@sh@@knotanchor{crossing}{8}
  \pgf@sh@@knotanchor{crossing}{16}
  \pgf@sh@@knotanchor{crossing}{32}
  \backgroundpath{
    \pgfutil@tempdima=\radius%
    \pgfmathsetlength{\pgf@xb}{\pgfkeysvalueof{/pgf/outer xsep}}%
    \pgfmathsetlength{\pgf@yb}{\pgfkeysvalueof{/pgf/outer ysep}}%
    \ifdim\pgf@xb<\pgf@yb%
      \advance\pgfutil@tempdima by-\pgf@yb%
    \else%
      \advance\pgfutil@tempdima by-\pgf@xb%
    \fi%
  }
}

\definecolor{darkblue}{rgb}{0,0,0.6}

\definecolor{amaranth}{rgb}{0.9, 0.17, 0.31} 
\definecolor{carrotorange}{rgb}{0.93, 0.57, 0.13} 
\definecolor{citrine}{rgb}{0.89, 0.82, 0.04} 
\definecolor{Green}{rgb}{0.05, 0.5, 0.06} 
\definecolor{teal}{rgb}{0.0, 0.5, 0.5} 
\definecolor{ballblue}{rgb}{0.13, 0.67, 0.8} 
\definecolor{Cerulean}{rgb}{0.16, 0.32, 0.75} 
\definecolor{amethyst}{rgb}{0.6, 0.4, 0.8} 
\definecolor{amber}{rgb}{1.0, 0.75, 0.0} 
\definecolor{burlywood}{rgb}{0.87, 0.72, 0.53} 
\definecolor{Gray}{rgb}{0.827,0.827,0.827}

\AtBeginDocument{%
   \def\MR#1{}
}

\title{On $\mathcal{T}$-positive links}
\author{Benjamin Bode}
\address{Departamento de Matemática Aplicada a la Ingeniería Industrial - ETSIDI, Universidad Politécnica de Madrid, Rda. de Valencia 3, Arganzuela, 28012 Madrid, SPAIN}
\email{benjamin.bode@upm.es}
\author{Paula Truöl}
\address{School of Mathematics and Statistics, University of Glasgow, University Place, Glasgow, G12 8QQ, United Kingdom}
\email{paula.truol@glasgow.ac.uk,
paulagtruoel@gmail.com}

\begin{document}

\makeatletter
\@namedef{subjclassname@2020}{%
	\textup{2020} Mathematics Subject Classification}
\makeatother



\begin{abstract} 
$\T$-positive links form a subset of strongly quasipositive links that strictly contains the set of all non-split braid positive links. Analogous to Baader's characterisation of positive links as precisely the strongly quasipositive and homogeneous links, we show that $\T$-positive links are precisely the strongly quasipositive links that are the closures of $\T$-homogeneous braids. This complements previous characterizations of $\T$-positive links by Rudolph and Banfield as links arising as boundaries of positive Hopf-plumbed baskets, or closures of staircase braids. We examine the behavior of $\T$-positive links under cabling operations and connected sums, and demonstrate that all strongly quasipositive, fibered knots with at most 12 crossings are $\T$-positive. Additionally, we compare $\T$-positivity with other positivity notions for links and compile open questions.
\end{abstract}

\maketitle

\section{Introduction} 

$\mathcal{T}$-positive links were defined by Rudolph~\cite{Rudolph_Hopf_plumbing} over twenty years ago. They are, in particular, \emph{strongly quasipositive}, so they arise as closures of strongly quasipositive braids. Strongly quasipositive braids also first appeared in the work of Rudolph~\cite{Rudolph_1983,rudolph_linkpolys} and can be best defined using the \emph{BKL-generators} (or \emph{band generators})~$a_{i,j}$ of the braid group $B_n$ on $n$ strands~\cite{birmankolee}. In terms of the classical Artin generators~\cite{artin_1925} of $B_n$, denoted $\sigma_1, \dots, \sigma_{n-1}$, the BKL-generators of~$B_n$ are given by 
\begin{align*}
\sigij = \left(\sigma_i \sigma_{i+1} \cdots \sigma_{j-2} \right) \sigma_{j-1} \left(\sigma_i \sigma_{i+1} \cdots \sigma_{j-2} \right)^{-1} 
\qquad \text{for} \quad 1 \leq i < j \leq n \qquad \text{(see \Cref{fig:sqpbraidgen})}.
\end{align*}

\begin{figure}[htbp] 
	 \centering
     \def\svgwidth{0,23\columnwidth}
	 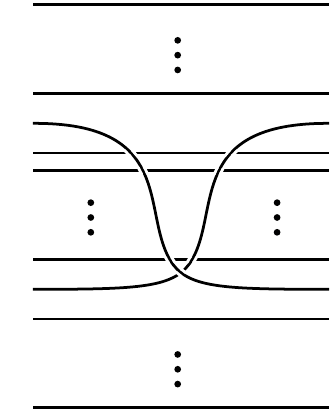
	\caption{The BKL-generator $\sigij$.}
	\label{fig:sqpbraidgen}
\end{figure}

We refer to \Cref{subsubsec:braids_present} for the relations of $B_n$ in terms of the $a_{i,j}$.
A braid $B \in B_n$ is \emph{strongly quasipositive} or \emph{BKL-positive} if it is a product of positive BKL-generators (no inverses~$a_{i,j}^{-1}$), that is,
\begin{align*}
    B = \prod_{k=1}^m a_{i_k,j_k} \qquad \text{for some} \quad 1 \leq i_k < j_k \leq n, \quad k\in \{1,\dots,m\}.
\end{align*}
In \Cref{subsubsec:T-homo}, following~\cite{Rudolph_Hopf_plumbing}, we define the set of \emph{$T$-generators} corresponding to the $n-1$ edges of an \emph{espalier}~$T$, which is a planar tree with $n$ vertices. These $T$-generators form a subset of cardinality~$n-1$ of the set of~$\binom{n}{2}$ BKL-generators $\{a_{i,j}\}_{1 \leq i < j \leq n}$. A braid~$B \in B_n$ is called \emph{$T$-positive} if it is a product of positive $T$-generators. For technical reasons, we also assume that each of the~$n-1$ generators specified by~$T$ appears at least once in~$B$. We refer to \Cref{subsubsec:T-homo} for the precise definitions. A link is \emph{$\mathcal{T}$-positive} if it arises as the closure of a $T$-positive braid for some espalier~$T$. By definition, $\mathcal{T}$-positive links are strongly quasipositive. In fact, there are inclusions
\begin{align}\label{eq:inclusions}
    \begin{split}
        \{\text{non-split braid positive links}\} &\subsetneq \{\text{$\mathcal{T}$-positive links}\} \\&\subsetneq \{\text{strongly quasipositive links}\}.
    \end{split}
\end{align}
Recall that a link is called \emph{braid positive} if it arises as the closure of a positive braid. A braid~$B \in B_n$ is \emph{positive} if it is a product of positive BKL-generators $a_{i,i+1}=\sigma_i$ for~$1 \leq i < n$. Using a specific espalier, namely the \emph{linear graph} $T_n$ with vertices at the first $n$ integers on the positive real axis in $\R^2$ and edges between~$i$ and~$i+1$ for each~$1 \leq i < n$ (see \Cref{fig:espalierT_n}), these are precisely the $T_n$-generators. This implies the first of the inclusions in \eqref{eq:inclusions}; see \Cref{subsubsec:T-homo} for more details. 

 \begin{figure}[htbp] 
     \centering
     \def\svgwidth{0,35\columnwidth}
	 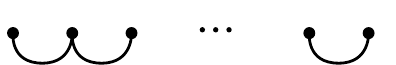 
	\caption{An example of an espalier: the linear graph $T_n$.}
    \label{fig:espalierT_n}
\end{figure}
 
In analogy to known results for positive and braid positive links (see below), we show the following. 

\begin{theorem}\label{theorem:main_T-pos}
$\T$-positive links are precisely the strongly quasipositive links that are $\T$-homogeneous.
\end{theorem}

\emph{$\T$-homogeneous links} are defined as the closures of $T$-homogeneous braids. For an espalier $T$, the \emph{$T$-homogeneous braids} are represented by braid words $B$ in the $T$-generators such that each $T$-generator appears at least once and only with all exponents of the same fixed sign in~$B$. \emph{Homogeneous} braids in~$B_n$ in the sense of Stallings~\cite{stallings_1978} are precisely the $T_n$-homogeneous braids. 

\begin{corollary}\label{prop:braid_positive}
Non-split braid positive links are precisely the strongly quasipositive links that are closures of homogeneous braids. 
\end{corollary}

We believe that \Cref{prop:braid_positive} 
might be well-known to experts, but to the best of our knowledge it has not appeared in this form in the literature. The slightly weaker analogous statement for braid positive diagrams was already observed by Baader~\cite[p.~287]{baader_pos}, who also showed the following.

\begin{theorem}[{\cite{baader_pos}; see also \cite{abe_homog,abe-tagami,FLL:almostpos}}]\label{theorem:pos_Seb}
Positive links are precisely the strongly quasipositive links that are homogeneous.
\end{theorem}

Here, \emph{positive} links are those links that admit a diagram with only positive crossings. We refer to~\cite{baader_pos} for the definition of homogeneous links. 

\subsection{Staircase braids and cabling}

$\T$-positive links have previously been characterized by Rudolph~\cite{Rudolph_Hopf_plumbing} as those links that arise as the boundary of a positive Hopf-plumbed basket. That is, they have a Seifert surface that can be obtained from a disk~$D$ by iteratively plumbing positive Hopf bands, with each plumbing arc contained within the original disk $D$. See \Cref{sec:prelims}, particularly \Cref{theorem:Hopf_plumbed} in \Cref{subsubsec:baskets}, for more details. 
Banfield~\cite{banfield} provided yet another characterization of $\T$-positive links. 
A strongly quasipositive braid in~$B_n$ is called a \emph{staircase braid} if it contains a positive power of the dual Garside element $$\delta = \sigma_1 \sigma_2 \dots \sigma_{n-1} \in B_n.$$ Specifically, a staircase braid can be written as a product of positive BKL-generators containing the subword $\delta$; or, in other words, the exponent of $\delta$ in the dual Garside normal form of the braid (see~\cite[Theorem~3.10]{birmankolee}) is strictly positive.

\begin{theorem}[{\cite[Corollary 5.11]{banfield}}]\label{thm:banfield}
    A link is $\mathcal{T}$-positive if and only if it arises as the closure of a staircase braid. 
\end{theorem}




In light of \Cref{thm:banfield}, 
we can use staircase braids to study the effect of cabling on $\T$-positive knots. Given a knot $K$, its \emph{$(p, q)$-cable $K_{p,q}$} for coprime integers $p,q \geq 1$ is the satellite knot with pattern the torus knot $T_{p,q}$ and companion $K$. See, for example, \cite[Section 4D]{rolfsen_2003}. 

\begin{proposition}\label{prop:cables}
Let $p \geq 2$ and let $K$ be a knot which is the closure of a staircase braid $B = \delta P\in B_n$ for some BKL-positive word $P$. If $q \geq n$, then the cable knot~$K_{p,q}$ is represented by a staircase braid.
\end{proposition}

Sufficient conditions like the one in \Cref{prop:cables} for a cable knot to fulfill a certain positivity notion are fairly well understood for strong quasipositivity, positivity and braid positivity. See \cite[Theorem~5.1]{kmms} for a concise summary. In particular, if $K$ satisfies the conditions of \Cref{prop:cables}, then the cable knot~$K_{p,q}$ is strongly quasipositive and fibered if and only if $q \geq 1$~\cite{hedden:cabling}. Thus~$q \geq 1$ is a necessary condition in \Cref{prop:cables}. We ask whether 
this can be upgraded to~$q \geq n$. 

\begin{question}
Is the ``only if'' direction of \Cref{prop:cables} true as well? That is, if $p \geq 2$ and $K$ is the closure of a staircase braid on $n$ strands such that $K_{p,q}$ is represented by a staircase braid as well, must~$q \geq n$?
\end{question}

The positive trefoil $T_{2,3}$ gives an example of a knot where the condition $q \geq n$ in \Cref{prop:cables} is in fact necessary; see \Cref{ex:cable_converse} for details. 
%
The equivalent characterizations of knots that are closures of staircase braids due to \cite{Rudolph_Hopf_plumbing,banfield} lead to versions of \Cref{prop:cables} for knots presented as the boundary of a positive Hopf-plumbed basket or as the closure of a $\T$-positive braid. For instance, if a knot~$K$ is the boundary of a positive Hopf-plumbed basket obtained from a disk by plumbing~$m$ positive Hopf bands, then there exists an associated staircase braid representative on $2m$ strands (see \cite[Proof of Lemma~5.7]{banfield}). According to \Cref{prop:cables}, the cable knot~$K_{p,q}$ is a staircase braid closure if~$p \geq 2$ and~$q \geq 2m$. By~\cite{Rudolph_Hopf_plumbing}, these cables arise as boundaries of positive Hopf-plumbed baskets. 

\subsection{\texorpdfstring{$\T$-positive}{T-positive} knots with small crossing number}

$\T$-positive knots are strongly quasipositive and fibered, by definition and by Rudolph's characterization as boundaries of positive Hopf-plumbed baskets~\cite{Rudolph_Hopf_plumbing}, respectively; see also \Cref{fiber_surf,Murasugi_sums}. We show that $\T$-positive knots are abundant among low-crossing-number knots in the following sense. 
\begin{proposition}\label{prop:sqp_fib_12_cross}
    All 42 strongly quasipositive, fibered knots with at most 12 crossings are $\T$-positive.
\end{proposition}

Since positive knots are strongly quasipositive \cite{Rudolph_positiveLinksSQP,nakamura}, as a consequence of \Cref{prop:sqp_fib_12_cross}, all 33 positive, fibered knots with at most 12 crossings are $\T$-positive.
In this context, we would like to highlight the following interesting open problem.

\begin{question}\label{question:pos_fib_T-pos}
    Are there positive, fibered knots which are not $\mathcal{T}$-positive? 
\end{question}

We discuss more on the context of \Cref{question:pos_fib_T-pos} in \Cref{ex:T-pos_and_pos}. In \Cref{cor:counterex_br_index}, we observe that there are no knots with braid index $2$ or $3$ that provide counterexamples to \Cref{question:pos_fib_T-pos}.\\

\textbf{Organization.} 
In \Cref{sec:prelims}, we review the basics of braids, braided Seifert surfaces, quasipositive surfaces, fiber surfaces, and Murasugi sums; make the definitions of $\T$-homogeneous and $\T$-positive braids more precise (\Cref{subsubsec:T-homo}); and compare $\T$-positivity to other positivity notions (\Cref{subsec:comparison}). \Cref{sec:characterisations} proves \Cref{theorem:main_T-pos} and \Cref{prop:braid_positive}, while \Cref{prop:cables} is proved in \Cref{sec:cabling}. \Cref{sec:crossing_nr} contains the proof of \Cref{prop:sqp_fib_12_cross}. \Cref{sec:further_results} discusses the relationship between $\T$-positivity and the unknotting number, the braid index, connected sums, visual primeness and positive trefoil plumbings. We end with two flowcharts in \Cref{fig:flowchart,fig:flowchart2} (\Cref{subsec:flowchart}) that summarize the various known implications and non-implications between the positivity notions studied in this paper.\\

\textbf{Acknowledgements.}
BB would like to thank Mikami Hirasawa for helpful discussions. PT would like to express her gratitude to her mathematical family, especially Sebastian Baader, Peter Feller, Livio Ferretti, Lukas Lewark, and Filip Misev, for teaching her the beauty of fiber surfaces. She would also like to thank Mark Kegel, Naageswaran Manikandan, and Diego Santoro for their interest and helpful discussions, as well as Arunima Ray for sharing her inspiring love of flowcharts.\\

\textbf{Grant support.} PT acknowledges support by the Swiss National Science Foundation Postdoc.Mobility fellowship 230329. She would like to thank the University of Glasgow for their hospitality and support. Part of this project was carried out during the first author's visit to and the second author's stay at the Max Planck Institute for Mathematics in Bonn. The authors would like to thank the institute for its hospitality and support.

\section{Preliminaries}\label{sec:prelims}

Let us make the definitions from the introduction more precise. Throughout, we will assume familiarity with (geometric) braids and their closures; see \eg \cite{birmanbrendle} for an introduction.

\subsection{BKL-generators of the braid group}\label{subsubsec:braids_present}

The braid group $B_n$ has a classical presentation in terms of $n-1$ generators $\sigma_1, \dots, \sigma_{n-1}$ which is due to Artin~\cite{artin_1925}. For our purposes, a different presentation due to Birman--Ko--Lee~\cite{birmankolee} is more convenient. In this presentation, there are generators~$a_{i,j}$ for~$1 \leq i < j \leq n$, called \emph{BKL-generators} or \emph{band generators}. The relations among them are 
\begin{align}
    a_{i,j}a_{k,\ell} &= a_{k,\ell} a_{i,j} && \text{if } (i-k)(i-\ell)(j-k)(j-\ell) > 0, \label{rel:1}\\
    a_{i,j}a_{j,k} &= a_{i,k}a_{i,j} = a_{j,k}a_{i,k} && \text{for all } i,j,k \text{ with } 1 \leq i < j < k \leq n.\label{rel:2}
\end{align}
In terms of the Artin generators $\sigma_i$, the BKL-generators can be expressed as
\begin{align*}
a_{i,j} = \left(\sigma_i \sigma_{i+1} \cdots \sigma_{j-2} \right) \sigma_{j-1} \left(\sigma_i \sigma_{i+1} \cdots \sigma_{j-2} \right)^{-1}  \quad \text{and} \quad a_{i,i+1} = \sigma_i.
\end{align*}

Words in the BKL-generators represent geometric braids. \Cref{fig:sqpbraidgen} shows the geometric braid which corresponds to the BKL-generator $\sigij$. We will refer to words in the BKL-generators as \emph{BKL-words}. 

\subsection{Braided Seifert surfaces}\label{subsubsec:braided_surf}

A \emph{Seifert surface} (for a link $L$) is an oriented, compact surface in~$S^3$ with no closed components (and oriented boundary~$L$). To each BKL-word $B \in B_n$, we can associate in a straightforward way a \emph{braided Seifert surface} in the sense of Rudolph; see \cite[Section 5.4]{Rudolph_Hopf_plumbing} and \cite[Section 2]{rudolph_braidedsurfaces}. Starting from~$n$ parallel disks, the Seifert surface is obtained by inserting a half-twisted band between the $i$th and $j$th disk for each BKL-generator $\sigij^{\pm 1}$ in the given word $B$, where the sign of the half-twist corresponds to the sign of the generator. Note that the topology of this surface does not change under the BKL-relations~\eqref{rel:1} and~\eqref{rel:2} (see \eg \cite[Figure~1]{birmankolee}). However, it does change under the trivial group relations $\sigij\sigij^{-1} = \sigij^{-1}\sigij = e$. We denote the braided Seifert surface associated to the braid word~$B$ by~$F(B)$. For an example, see \Cref{fig:seifsurf_T-homo} (bottom) for the braided Seifert surface associated to the BKL-word $B = a_{1,3}^2 a_{2,3}^2 a_{4,5}^2 a_{1,4}^{-3} a_{4,5}^2 a_{2,3} a_{1,3} a_{4,5}$ representing the geometric braid in \Cref{fig:seifsurf_T-homo} (top).
By slight abuse of notation, we often identify braid words and the corresponding geometric braids.

\begin{figure}[htbp] 
     \centering
     \def\svgwidth{0,9\columnwidth}
	 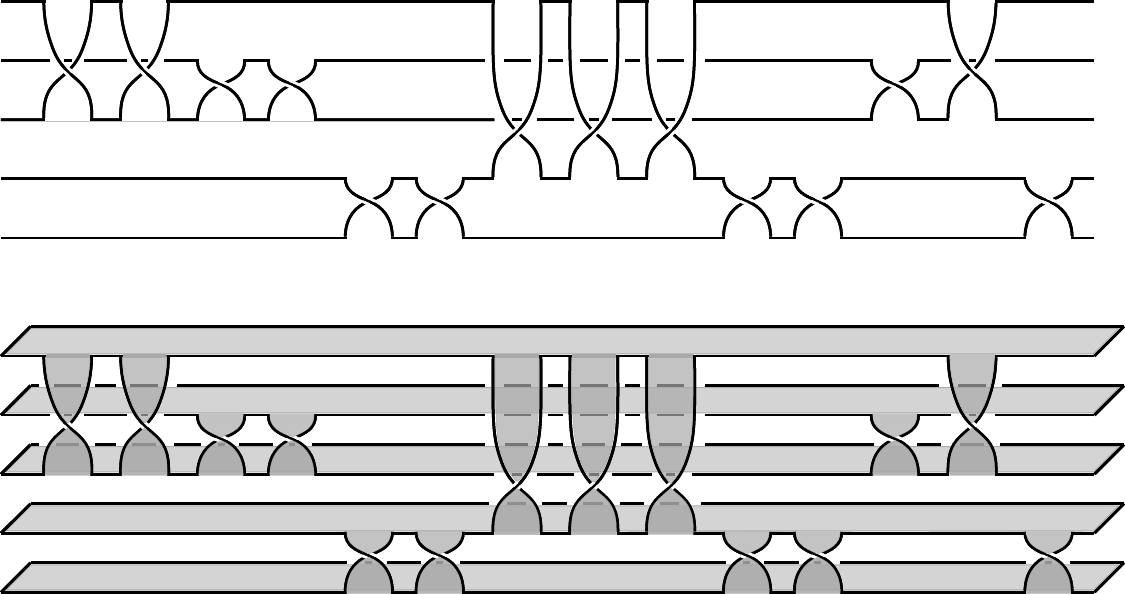 
	\caption{$\T$-homogeneous braid $B = a_{1,3}^2 a_{2,3}^2 a_{4,5}^2 a_{1,4}^{-3} a_{4,5}^2 a_{2,3} a_{1,3} a_{4,5}$ (top) and the associated braided Seifert surface $F(B)$ (bottom).}
    \label{fig:seifsurf_T-homo}
\end{figure}

\subsection{Quasipositive surfaces}\label{sec:qp}

Let us consider a special case of braided Seifert surfaces; see also~\cite{rudolph_braidedsurfaces,Rudolph_surfaces}. Suppose that a link~$L$ arises as the closure of a BKL-positive braid~$A = \prod_{k=1}^\ell a_{i_k, j_k} \in B_m$ for some $1 \leq i_k < j_k \leq m$. Then the braided surface~$F(A)$ associated with $A$ 
consists of $m$ disjoint parallel disks and~$\ell$~positively half-twisted bands connecting these disks. According to work of Bennequin~\cite{bennequin}, the Euler characteristic $\chi(F(A))=m-\ell$ of $F(A)$ realizes the maximal Euler characteristic $\chi(L)$ among all Seifert surfaces for~$L$. Any Seifert surface $F$ for $L$ is called \emph{quasipositive} if it is ambient isotopic (in $S^3$) to $F(A)$ for some strongly quasipositive braid $A$.

\subsection{\texorpdfstring{$\T$-homogeneous}{T-homogeneous} and \texorpdfstring{$\T$-positive}{T-positive} braids}\label{subsubsec:T-homo}

Following \cite{Rudolph_Hopf_plumbing}, an \emph{espalier} is a tree $T$ (a finite connected graph with no cycles) which can be embedded in the lower half-plane $\{y \leq 0\}\subset\R^2$ such that the $n \geq 1$ vertices of $T$ get mapped to the points $(j,0)\in \R^2$ for $j \in \{1, \dots,n\}$ and the interiors of the $n-1$ edges of $T$ get mapped disjointly to~$\{y < 0\}$. For an espalier~$T$, we denote by $V(T)=\{1, \dots,n\}$ and~$E(T)$ its vertex and edge set, respectively. To each edge $(i,j)\in E(T)$ between vertices $i,j \in V(T)$, we associate the BKL-generator $\sigij \in B_n$. As in \cite{banfield}, we denote by $G(T)=\{\sigij\}_{(i,j) \in E(T)}$ the set of all BKL-generators associated to edges of $T$, and call them the 
 \emph{$T$-generators}. Three examples of espaliers with associated generating sets are shown in \Cref{fig:espaliers}.

\vspace{1em}

\begin{figure}[htbp] 
	 \centering
     \def\svgwidth{1\columnwidth}
	 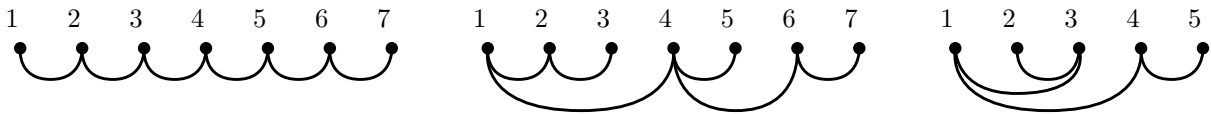
	\caption{From left to right: the linear graph $T_7$ and sample espaliers~$T_m$ and~$T_r$ on seven and five vertices, respectively. The associated generating sets for the espaliers are $G(T_7)=\{a_{1,2}, a_{2,3}, a_{3,4}, a_{4,5}, a_{5,6}, a_{6,7}\}$, $G(T_m)=\{a_{1,2}, a_{2,3}, a_{1,4}, a_{4,5}, a_{4,6}, a_{6,7}\}$ and $G(T_r)=\{a_{1,3}, a_{1,4}, a_{2,3}, a_{4,5}\}$.}
    \label{fig:espaliers}
\end{figure}

We call a braid~$B$ on $n$ strands \emph{$T$-homogeneous} if it can be represented by a word in the $T$-generators such that for all $(i,j) \in E(T)$, the $T$-generator~$\sigij$ appears at least once in this braid word and with either only positive or only negative powers. Moreover, $B$ is called \emph{$T$-positive} if all appearances of any of the $\sigij$ for $(i,j) \in E(T)$ are with only positive powers. Having made these definitions more precise, let us recall the following definitions from the introduction. A link is \emph{$T$-homogeneous} or \emph{$T$-positive} if it arises as the closure of a $T$-homogeneous or $T$-positive braid, respectively. Moreover, a link is $\T$-homogeneous or $\T$-positive if it is \emph{$T$-homogeneous} or \emph{$T$-positive} for some espalier~$T$, respectively. 

\begin{example}
    \Cref{fig:seifsurf_T-homo} (top) shows the $T_r$-homogeneous braid $B = a_{1,3}^2 a_{2,3}^2 a_{4,5}^2 a_{1,4}^{-3} a_{4,5}^2 a_{2,3} a_{1,3} a_{4,5}$ for the espalier $T_r$ from \Cref{fig:espaliers} (right). 
\end{example}

\begin{example}\label{ex:bp_T_n}
For the linear graph $T_n$ in $\R^2$ with vertices at the first $n$ integers on the positive real axis and edges from~$i$ to~$i+1$ for each $1 \leq i < n$ (see \Cref{fig:espalierT_n} and \Cref{fig:espaliers} (left)), the $T_n$-generators are precisely the BKL-generators $a_{i,i+1}=\sigma_i$, $1 \leq i < n$, so the classical Artin generators of $B_n$. Thus the $T_n$-positive braids are precisely the positive braids in~$B_n$ where each Artin generator $a_{i,i+1}=\sigma_i$ appears at least once. We call such positive braids \emph{strictly positive}. 
\end{example}

Every non-split braid positive link can be represented by a strictly positive braid, which, by \Cref{ex:bp_T_n}, implies the first of the inclusions in \eqref{eq:inclusions}. Since non-split braid positive links are $\mathcal{T}$-positive, the well-known families of \emph{positive torus links}~$T_{p,q}$, which are the closures of $(\sigma_1 \cdot \ldots \cdot \sigma_{p-1})^q$ for positive integers~$p, q \geq 1$, and \emph{algebraic links}, which are links of isolated singularities of complex algebraic plane curves~\cite{milnor_book}, are also $\mathcal{T}$-positive. 

\subsection{Fiber surfaces}\label{fiber_surf}

A Seifert surface $F$ is a \emph{fiber surface} and $\partial F = L$ a \emph{fibered} link if there is a fiber bundle $\varphi \colon S^3 \setminus \partial F \to S^1$ whose fibers are isotopic to the interior of $F$. Fiber surfaces are connected, incompressible, and realize the maximal Euler characteristic~$\chi(L)$ among Seifert surfaces for~$L$~\cite{stallings_1978}. Moreover, up to ambient isotopy a fiber surface~$F$ is the unique incompressible (and thus maximal Euler characteristic) surface with boundary $L = \partial F$. See~\cite[Proposition 2.19]{rudolph_handbook}, which cites \cite{stallings_1978,gabai_sum_I,gabai:foliations,gabai:arborescent}.

\subsection{Murasugi sums}\label{Murasugi_sums}

Finally, let us recall an---as Gabai phrases it \cite{gabai_sum_I,gabai_sum_II}---natural geometric operation on Seifert surfaces. Given two Seifert surfaces $F_1$ and $F_2$, their \emph{Murasugi sum}~$F_1 \ast F_2$ is obtained by first identifying via orientation-preserving diffeomorphisms some even-sided polygon $P$ with subsets of~$F_1$ and~$F_2$ such that for both~$F_i$, $i \in \{1,2\}$, the edges of~$P$ alternate between boundary arcs and proper arcs in $F_i$, and then gluing $F_1$ and~$F_2$ along an orientation-reversing diffeomorphism along $P$ (which identifies the proper arcs in $F_1$ with boundary arcs in $F_2$ and vice versa). See \eg \cite[Chapter~4.2]{kawauchi_book} for a precise definition.

\begin{proposition}[{\cite{stallings_1978,gabai_sum_I,gabai_sum_II}; see also 
\cite[Theorem 3.1]{gabai_det_fibred}}]\label{prop:Murasugi_sum_fib}
    Any Murasugi sum $F_1 \ast F_2$ of Seifert surfaces~$F_1$ and $F_2$ is a fiber surface (see \Cref{fiber_surf}) if and only if the two summands are. 
\end{proposition}

\begin{proposition}[{\cite{Rudolph_V}}]\label{prop:Murasugi_sum_qp}
    Any Murasugi sum $F_1 \ast F_2$ of Seifert surfaces $F_1$ and~$F_2$ is quasipositive (see \Cref{sec:qp}) if and only if the two summands are.
\end{proposition}

A direct corollary of \Cref{prop:Murasugi_sum_fib,prop:Murasugi_sum_qp} is that the link $\partial(F_1 \ast F_2)$ is strongly quasipositive or fibered if and only if both of the links $\partial F_1$ and $\partial F_2$ are strongly quasipositive or fibered, respectively. A special case of Murasugi sums is \emph{plumbing}, when the polygon~$P$ is a rectangle, and \emph{positive} or \emph{negative} \emph{Hopf plumbing}, when in addition one of the involved Seifert surfaces is a fiber surface for the torus link~$T_{2,\pm 2}$, respectively. 

\subsection{Comparison to other positivity notions}\label{subsec:comparison}

In this subsection, we 
compare $\T$-positivity to other notions of link positivity, many of which we have already encountered above. 
\Cref{fig:flowchart,fig:flowchart2} in \Cref{subsec:flowchart} summarize the various positivity notions and their relations.

\subsubsection{Positive Hopf-plumbed baskets}\label{subsubsec:baskets}
A Seifert surface is called a \emph{Hopf-plumbed basket} if it can be obtained from a disk~$D$ by iteratively plumbing Hopf bands such that each plumbing arc is contained in the initial disk $D$. 
It is a \emph{positive Hopf-plumbed basket} if all involved Hopf bands are positive. 
As a consequence of \Cref{prop:Murasugi_sum_fib}, Hopf-plumbed baskets, as special cases of iterative Hopf plumbings, are fiber surfaces and their boundaries are fibered links.

\begin{theorem}[{\cite{Rudolph_Hopf_plumbing}}]\label{theorem:Hopf_plumbed}
    A link is $\mathcal{T}$-positive if and only if arises as the boundary of a positive Hopf-plumbed basket. In particular, $\mathcal{T}$-positive links are fibered.
\end{theorem}

The various characterizations of $\T$-positive links can be summarized as follows: 
    \begin{align}\label{eq:summary}
        \begin{split}
            & \{ \text{strongly quasipositive links}\} \cap \{\T\text{-homogeneous links} \}         \\         \stackrel{\text{\Cref{theorem:main_T-pos}}}{=} &\{\T\text{- positive links} \}
            \\\stackrel{\text{\Cref{thm:banfield}}}{=}& \{ \text{links that are closures of staircase braids} \} \\
            \stackrel{\text{\Cref{theorem:Hopf_plumbed}}}{=} 
             & \{ \text{links that are boundaries of positive Hopf-plumbed baskets}\}. 
        \end{split}
    \end{align}

\subsubsection{\texorpdfstring{$\T$-positivity}{T-positivity} and positivity}\label{ex:T-pos_and_pos}
Let us provide two examples of knots which are $\T$-positive, but not positive. The mirrors of the knots $10_{145}$ and $12n_{642}$, denoted $K_1 \defeq m(10_{145})$ and $K_2 \defeq m(12n_{642})$, are strongly quasipositive and fibered, but not positive \cite{knotinfo}. The latter follows from a computation of the Conway polynomials of the knots, which by a result of Cromwell~\cite[Corollary~5.1]{cromwell:homo} shows that~$K_1$ and $K_2$ are not even homogeneous. Both knots are however $\mathcal{T}$-positive by \Cref{theorem:Hopf_plumbed}. Indeed, descriptions as positive Hopf-plumbed baskets for Seifert surfaces for $K_1$ and $K_2$ can be read off from \Cref{fig:fiber_surf}. Alternatively, $\T$-positivity of $K_1$ was already observed by Banfield; see \cite[Lemma~5.14]{banfield}. 
\begin{figure}[htbp] 
	 \centering
     \includegraphics[width=1\textwidth]{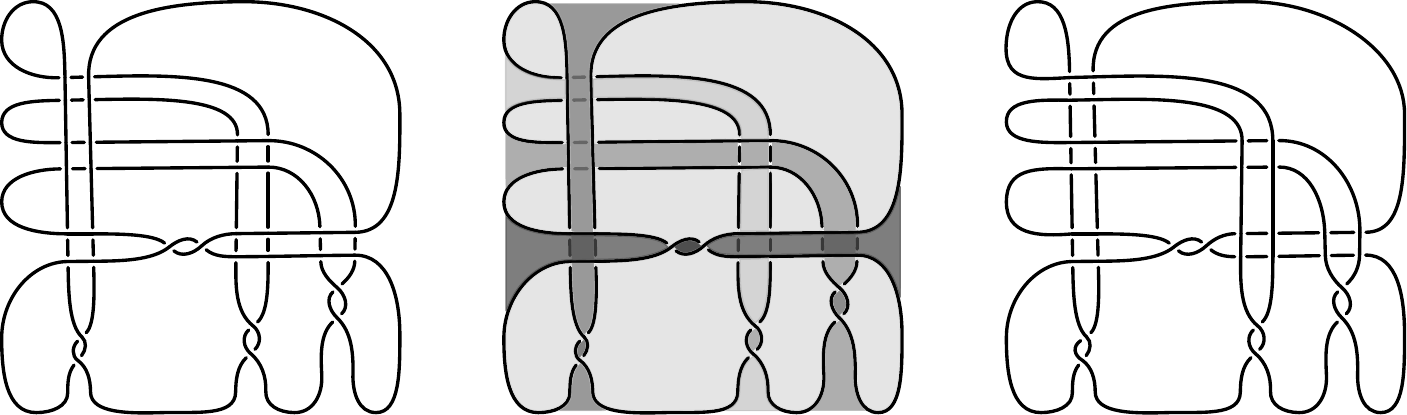}
    \caption{The knot $K_1 = m(10_{145})$ (left) and its genus two Seifert surface as positive Hopf-plumbed basket (middle). The knot $K_2 =m(12n_{642})$ (right) admits a similar Seifert surface of genus two (not drawn) where only the order of plumbing of the four positive Hopf bands differs.
    }
    \label{fig:fiber_surf}
\end{figure}

Positive knots are not necessarily $\T$-positive either, as can easily be seen by considering positive knots that are not fibered. Examples of such knots are plenty~\cite{knotinfo}, the easiest in terms of crossing number being $5_2$. In this context, we would like to highlight the following interesting open problem. 

\begin{question}[{\Cref{question:pos_fib_T-pos}}]\label{question:pos_fib_T-pos_sec}
    Are there positive, fibered knots which are not $\mathcal{T}$-positive? Equivalently, are there positive, fibered knots which do not arise as the boundary of a positive Hopf-plumbed basket?
\end{question}

Positive, fibered knots arise as the boundaries of positive Hopf plumbings. 
This follows from work of Cromwell~\cite[Theorem~5 and its corollaries]{cromwell:homo}; see also \cite[Figure~2]{kegel2024unknottingfiberedpositiveknots}. There are examples of knots that are the boundary of a positive Hopf plumbing but not of a positive Hopf-plumbed basket; that is, they are not $\T$-positive. Indeed, there are infinitely many such knots, as we shall explain now. Recall that the \emph{genus} of a knot~$K$, $g(K)$, is the minimal genus of any Seifert surface for $K$. For every $g \geq 2$, Misev~\cite{misev} constructed infinitely many knots that arise as the boundary of a positive Hopf plumbing of genus $g$. In contrast, the number of Hopf-plumbed baskets of a given genus is finite, which implies the claim. 
However, \Cref{question:pos_fib_T-pos_sec} asks for something stronger to be true. 
In \Cref{prop:pos_fib_small_crossing_nr} in \Cref{sec:crossing_nr}, we show that there are no positive, fibered knots with at most~12 crossings which provide counterexamples to \Cref{question:pos_fib_T-pos_sec}, since they are all $\T$-positive.

%

\subsubsection{\texorpdfstring{$\T$-positivity}{T-positivity} and braid positivity}\label{T-pos_braid_pos}

We have seen in \Cref{ex:T-pos_and_pos} that there exist $\T$-positive knots that are not positive and thus not braid positive. The knots $\mathcal{K}_n$, $n \geq 1$, from \cite[Proposition~5.3]{kegel2024unknottingfiberedpositiveknots} provide an infinite family of knots that are positive and fibered, but not braid positive. These knots are indeed $\T$-positive. This can be seen from \Cref{theorem:Hopf_plumbed} using their Seifert surfaces given in the proof of~\cite[Proposition~5.3 (1)]{kegel2024unknottingfiberedpositiveknots}. In particular, using \cite[Figure~9]{kegel2024unknottingfiberedpositiveknots}, note that the Seifert surface for $\mathcal{K}_n$ can be obtained from a disk by first plumbing the $3 \cdot (2n+1)$ positive Hopf bands corresponding to the $2n+1$ copies of the tangle~$T$ from \cite[Figure~8]{kegel2024unknottingfiberedpositiveknots}, and then additionally plumbing $4n+1$ positive Hopf bands corresponding to the fiber surface of the torus knot $T_{2,4n+2}$. So the Seifert surface of~$\mathcal{K}_n$ is a positive Hopf-plumbed basket and $\mathcal{K}_n$ thus $\T$-positive by \Cref{theorem:Hopf_plumbed}. As a consequence, we obtain the following.
\begin{proposition}
\label{prop:inf_T-pos_not_bp}
    There are infinitely many $\T$-positive knots that are not braid positive.
\end{proposition}

\subsubsection{\texorpdfstring{$\T$-positivity}{T-positivity} and strong quasipositivity}\label{T-pos_sqp}

$\T$-positive knots are strongly quasipositive and fibered, by definition and \eg by \Cref{theorem:Hopf_plumbed}, respectively. However, the converse is not true. The $(2,1)$-cable of the positive trefoil, $(T_{2,3})_{2,1}$, provides an example.
Indeed, this knot is not the boundary of a positive Hopf-plumbed basket since by work of Melvin--Morton~\cite{melvin_morton} its fiber surface does not deplumb a Hopf band. 
In fact, there are infinitely many strongly quasipositive, fibered knots that are not $\T$-positive. The knots constructed in~\cite{misev} give rise to such an infinite family for every~$g \geq 2$, where $g$ is the genus of the knots; see also the discussion in \Cref{ex:T-pos_and_pos}.

\newpage

\section{Equivalent characterization of \texorpdfstring{$\T$-positive}{T-positive} links}\label{sec:characterisations}

In this section, we prove the following proposition.

\begin{proposition}\label{prop:main_T-pos}
Let $T$ be an espalier. Then the $T$-positive links are precisely the strongly quasipositive links that are $T$-homogeneous.
\end{proposition}

\Cref{prop:main_T-pos} directly implies \Cref{theorem:main_T-pos}, \ie that $\T$-positive links are precisely the strongly quasipositive links that are $\T$-homogeneous.

\begin{proof}[{Proof of \Cref{theorem:main_T-pos}}]
One direction easily follows from the definitions. For the other direction, suppose that $L$ is a strongly quasipositive and $T$-homogeneous link for some espalier $T$. Then, by \Cref{prop:main_T-pos}, $L$ is $T$-positive for the same espalier~$T$, so $\T$-positive by definition.
\end{proof}

\begin{proof}[{Proof of \Cref{prop:main_T-pos}}]
Again, one direction easily follows from the definitions. 

For the other direction, suppose that $L$ is a strongly quasipositive link which has a $T$-homogeneous braid representative $B$ for some espalier $T$. Suppose that $T$ has~$n$ vertices. Using the notation from \Cref{subsubsec:T-homo}, $B$ can be represented by a word in the $T$-generators contained in $G(T)=\{\sigij\}_{(i,j) \in E(T)}$ such that for all $(i,j) \in E(T)$, the generator~$\sigij$ appears at least once in this braid word and with either only positive or only negative powers. By slight abuse of notation, we denote this braid word by $B$ as well.

Consider~$F(B)$, the braided Seifert surface for $L$  associated to the braid word~$B$ (see \Cref{subsubsec:braided_surf}). Since $B$ is $T$-homogeneous, the surface~$F(B)$ is a fiber surface and hence realizes $\chi(L)$~\cite{Rudolph_Hopf_plumbing} (see also~\Cref{fiber_surf}). Indeed, we claim that $F(B)$ is isotopic to the surface obtained by an iterative Murasugi sum of Seifert surfaces for torus links $T_{2,t_k}$, which in particular implies its fiberedness by \Cref{prop:Murasugi_sum_fib}. More precisely, suppose that 
\begin{align}\label{eq:edges}
E(T)=\{(i_1,j_1), \dots, (i_{n-1},j_{n-1})\}
\end{align}
is the edge set of~$T$ where $1 \leq i_k < j_k \leq  n$ for all $k \in \{1, \dots, n-1\}$. Let $t_k$ denote the exponent sum of the $T$-generator $a_{i_k,j_k}$ in the braid word~$B$. Since~$B$ is $T$-homogeneous, the absolute value of $t_k$ equals precisely the number of half-twisted bands in $F(B)$ that connect the $i_k$-th and the $j_k$-th disk in $F(B)$. With this notation at hand, our more precise claim is as follows.\\

\emph{Claim 1:} The surface $F(B)$ is ambiently isotopic in $S^3$ to the surface obtained by an iterated Murasugi sum of Seifert surfaces $F_{2,t_k}$ for the torus links $T_{2,t_k}$, where~$k \in \{1, \dots, n-1\}$. 

\begin{proof}[Proof of Claim 1]
Readers familiar with fiber surfaces and Murasugi sums will find the statement clear; see, for example, \cite[Remark~0.9]{rudolph_new_knot_inv} and \cite{Rudolph_Hopf_plumbing}. Indeed, the argument is very similar to the one in~\cite[Theorem~2]{stallings_1978}, which is used to decompose Seifert surfaces for closures of homogeneous braids into iterative Murasugi sums. See \cite[Step 2 in Section 2 and the proof of Theorem~1.1]{feller_lewark_orbegozo} for more details and figures. The only difference in the $T$-homogeneous case is that the iteration should be started at one of the leaves of the tree $T$, which does not necessarily coincide with starting the iteration at the top or bottom disk of the Seifert surface~$F(B)$, as in the homogeneous case. \Cref{fig:seifsurf_T-homo2} sketches the decomposition of~$F(B)$ for the $T$-homogeneous braid $B = a_{1,3}^2 a_{2,3}^2 a_{4,5}^2 a_{1,4}^{-3} a_{4,5}^2 a_{2,3} a_{1,3} a_{4,5}$, illustrating the procedure. The details of a step-by-step proof in the general case are left to the reader.
\end{proof}

\begin{figure}[htbp] 
     \centering
     \def\svgwidth{1\columnwidth}
	 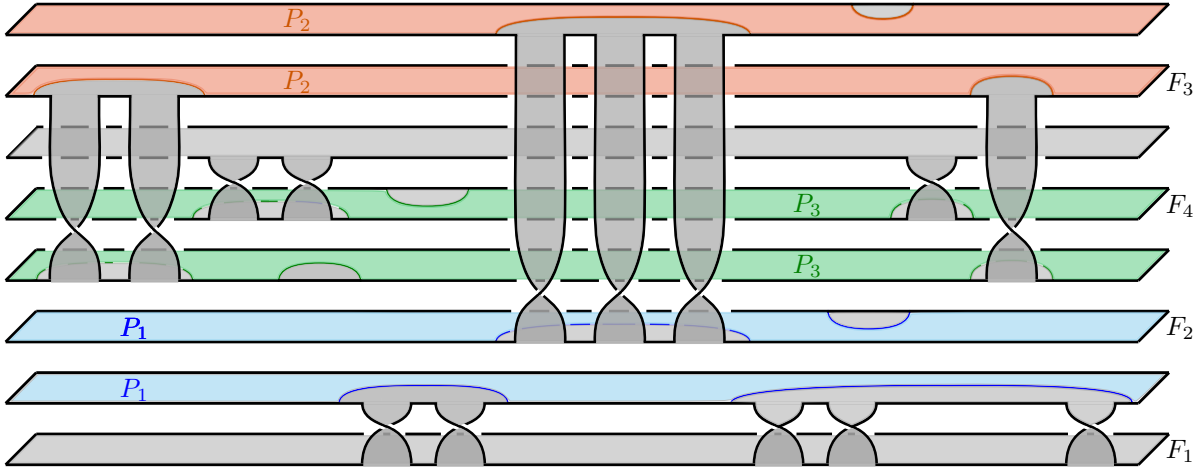 
	\caption{The fiber surfaces $F_1=F_{2,5}$, $F_2=F_{2,-3}$, $F_3=F_{2,3}$ and $F_4=F_{2,3}$ for the torus links $T_{2,5}$, $T_{2,-3}$, $T_{2,3}$ and $T_{2,3}$, respectively, together with summing regions $P_1$, $P_2$ and $P_3$ in blue, orange and green, respectively, such that Murasugi summing~$F_2$ to~$F_1$ using $P_1$, summing $F_3$ to the result using $P_2$, and summing~$F_4$ to that result via $P_3$ yields the braided Seifert surface~$F(B)$ from \Cref{fig:seifsurf_T-homo}.}
    \label{fig:seifsurf_T-homo2}
\end{figure}

Up to relabelling the edges in \eqref{eq:edges}, we can now 
assume that they are ordered such that 
\begin{align}\label{eq:Murasugi_sum_torus_links}
F(B) \cong (F_{2,t_1} \ast F_{2,t_2}) \ast \dots \ast F_{2,t_{n-1}}.
\end{align}

Next, since $L$ is strongly quasipositive, there is a strongly quasipositive braid~$A$ which determines a quasipositive Seifert surface~$F(A)$ for~$L$; see \Cref{sec:qp}. This surface $F(A)$ realizes $\chi(L)$~\cite{bennequin}. Combining the previous steps, there are two Seifert surfaces $F(A)$ and $F(B)$ for~$L$ which both realize~$\chi(L)$. However, up to isotopy the fiber surface $F(B)$ is the unique surface with maximal Euler characteristic bounding~$L$ (see~\cite[Proposition 2.19]{rudolph_handbook}), so~$F(A)$ and $F(B)$ must be ambiently isotopic. Thus~$F(B)$ is a quasipositive surface (by definition), which is obtained by an iterative Murasugi sum as in \eqref{eq:Murasugi_sum_torus_links}. By induction and \Cref{prop:Murasugi_sum_qp} due to~\cite{Rudolph_V}, all the summands $F_{2,t_k}$ must hence be quasipositive. This implies that $t_k \geq -1$ for all $k \in \{1,\dots,n-1\}$, since the torus link $T_{2,t_k} = \partial F_{2,t_k}$ is strongly quasipositive if and only if $t_k \geq -1$. The latter follows \eg from~\cite{Traczyk} using \Cref{theorem:pos_Seb}.

Let us now distinguish two cases. If $t_k \geq 0$ for all $k \in \{1,\dots,n-1\}$, then by definition of $t_k$, each $T$-generator~$a_{i_k,j_k}$ appears with only positive signs in $B$. Hence $B$ is a $T$-positive braid word and we are done. On the other hand, if $t_{\kappa} =-1$ for some $\kappa \in \{1,\dots,n-1\}$, then the $T$-generator $a_{i_\kappa,j_\kappa}$ has a single appearance as~$a_{i_\kappa,j_\kappa}^{-1}$ in the word $B$.\\

\emph{Claim 2:} The braid word $B^\prime$ obtained from $B$ by replacing the letter $a_{i_\kappa,j_\kappa}^{-1}$ by~$a_{i_\kappa,j_\kappa}$ represents a $T$-homogeneous braid with closure~$L$. In fact, the braided Seifert surfaces $F(B^\prime)$ and $F(B)$ are ambiently isotopic. 

\begin{proof}[Proof of Claim 2]
Since the $T$-generator $a_{i_\kappa,j_\kappa}$ appears in $B$ precisely once, as~$a_{i_\kappa,j_\kappa}^{-1}$, the surface~$F(B)$ becomes disconnected if we remove the corresponding negatively twisted band. Call the sub-surfaces of~$F(B)$ that correspond to these connected components $F_1$ and $F_2$, and the band $\beta$. We now modify~$F(B)$ by the following ambient isotopy: fix $F_1$, while rotating $F_2$ together with $\beta$ through a full rotation by angle $2 \pi$. The rotation axis passes through the band $\beta$ so that it is changed to a positively twisted band. See \Cref{fig:seifsurf_T-homo_2} for an example. 
The resulting surface is the one associated to $B^\prime$.
\end{proof}

\begin{figure}[htbp] 
    \centering
     \def\svgwidth{0.9\columnwidth}
	 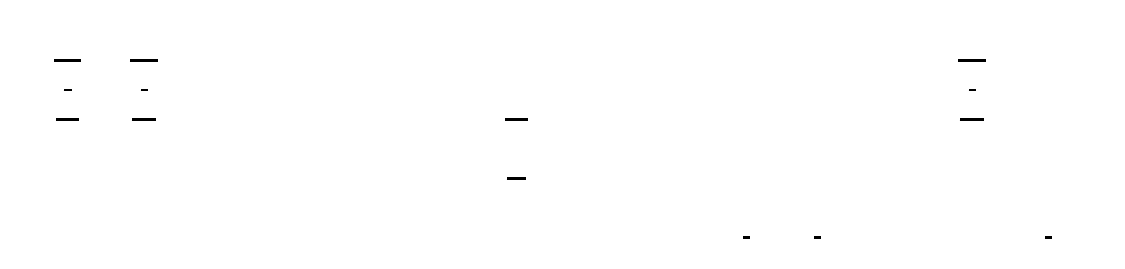 
	\caption{Isotopy used in the proof of Claim 2.}
    \label{fig:seifsurf_T-homo_2}
\end{figure}

By an induction on the number of indices $k \in \{1, \dots, n-1\}$ for which~$t_k = -1$ and iteratively applying Claim 2, we obtain a braid word where each $T$-generator $a_{i_k,j_k}$ appears with only positive exponents. This is a $T$-positive braid with closure~$L$, which finishes the proof.
\end{proof}

\begin{corollary}[{\Cref{prop:braid_positive}}]\label{cor:bp}
Non-split braid positive links are precisely the strongly quasipositive links that are closures of homogeneous braids. 
\end{corollary}

\begin{proof}[{Proof of \Cref{cor:bp}}]
    Let $L$ be a non-split braid positive link. Then $L$ can be represented by a strictly positive braid $B$ on, say, $n$ strands. By \Cref{ex:bp_T_n}, this braid is $T_n$-positive for the linear graph~$T_n$ (see \Cref{fig:espalierT_n}), so also strongly quasipositive and $T_n$-homogeneous, \ie homogeneous. Conversely, a homogeneous braid on $n$ strands is a $T_n$-homogeneous braid, and \Cref{prop:main_T-pos} implies that the strongly quasipositive closure of a $T_n$-homogeneous braid is $T_n$-positive, so strictly braid positive.
\end{proof}

A link is called \emph{almost positive} if it admits a diagram with a single negative crossing. Feller--Lewark--Lobb \cite{FLL:almostpos} showed that almost positive links are strongly quasipositive. 
Comparing the analogies in \Cref{theorem:main_T-pos},  \Cref{cor:bp} (\Cref{prop:braid_positive}) and \Cref{theorem:pos_Seb}, the following question seems natural to ask. 

\begin{question}
Can we characterize almost positive links as strongly quasipositive links with an additional property of being \emph{almost homogeneous}? 
\end{question}

\section{Staircase braids and cabling}\label{sec:cabling}

In this section, we prove the following proposition.

\begin{proposition}[{\Cref{prop:cables}}]\label{prop:staircase_cabling}
Let $p \geq 2$ and let $K$ be a knot which is the closure of a staircase braid $B = \delta P\in B_n$ for some BKL-positive word $P$. If $q \geq n$, then the cable knot~$K_{p,q}$ is represented by a staircase braid.
\end{proposition}

\begin{remark}
    It is well known that, given a strongly quasipositive, fibered knot $K$, the cable knot~$K_{p,q}$ is strongly quasipositive whenever~$q \geq 1$~\cite{hedden:cabling}. However, to the best of our knowledge, explicit strongly quasipositive braid representatives for~$K_{p,q}$ have not yet been documented in the literature. The proof below provides an algorithm for obtaining these representatives.
\end{remark}

\begin{proof}[{Proof of \Cref{prop:staircase_cabling}}]
    Throughout this proof, let $p \geq 2$. Suppose that $B = \delta P \in B_n$ is a staircase braid representative for $K$ where $P$ is a BKL-positive word. Recall that $$\delta_n \defeq \delta=\sigma_1 \sigma_2 \cdots \sigma_{n-1} \in B_n$$ is the dual Garside element. For every $q \geq n$, we claim that there is a staircase braid~$B_{p,q}$ on $pn$ strands with closure $K_{p,q}$, \ie $B_{p,q}$ is a BKL-positive word containing $\delta_{pn} \defeq \sigma_1 \sigma_2 \cdots \sigma_{pn-1}$. 

    Given a diagram $D$ for $K$, the \emph{standard diagram} $D_{p,q}$ for $K_{p,q}$ is obtained by taking $p$~parallel (blackboard) copies of each strand of $D$ and adding $pw-q$ negative $(1/p)$-twists to the $p$~parallel strands, where~$w$ is the writhe of $D$. See \eg \cite[Figure~7]{kmms} for an example. When applied to a braid diagram $D$ on $n$ strands, this procedure gives rise to a braid diagram $D_{p,q}$ on $pn$ strands representing~$K_{p,q}$.  
    
    Starting from the braid diagram $D$ on $n$ strands associated to the staircase braid word~$B = \delta P$, we first claim that the standard braid diagram $D_{p,0}$ is BKL-positive. Indeed, after a braid isotopy, each BKL-generator in~$B$ translates to a product of~$p$ many BKL-generators in the $(p,0)$-cable~$D_{p,0}$. More precisely, the BKL-generator~$\sigij$ becomes 
    \begin{align}\label{eq:cabling}
        a_{pi,pj} a_{pi-1,pj-1}\cdots a_{p(i-1)+1,p(j-1)+1} = \prod_{k=0}^{p-1} a_{pi-k,pj-k}.
    \end{align}
    See \Cref{fig:cabling_BKL-gen} for an example in which the isotopy is depicted. We have also drawn fence diagrams as short-hand notation corresponding to this isotopy. Although it seems self-explanatory, let us note that each vertical line in a fence diagram stands for a positive BKL-generator. In \Cref{fig:cabling_BKL-gen_fence}, we have drawn the fence diagrams in the general case. In particular, if $B$ is a product of~$m$ positive BKL-generators, then~$B_{p,0}$ is a product of~$pm$ positive BKL-generators, where the product splits into~$m$ factors of the form given in~\eqref{eq:cabling}, one such factor for each BKL-generator in $B$. This shows that $K_{p,0}$ has a BKL-positive braid representative~$B_{p,0}$ in~$B_{pn}$. The braid diagram $B_{p,q}$ for $q > 0 $ is obtained from $B_{p,0}$ by inserting~$q$ positive $(1/p)$-twists on a set of~$p$ parallel strands in $B_{p,0}$, \ie by inserting $q \cdot (p-1)$ positive BKL-generators. So if $B$ is BKL-positive, then $B_{p,q}$ is BKL-positive, too.

    \begin{figure}[htbp] 
        \centering
        \def\svgwidth{1\columnwidth}
	       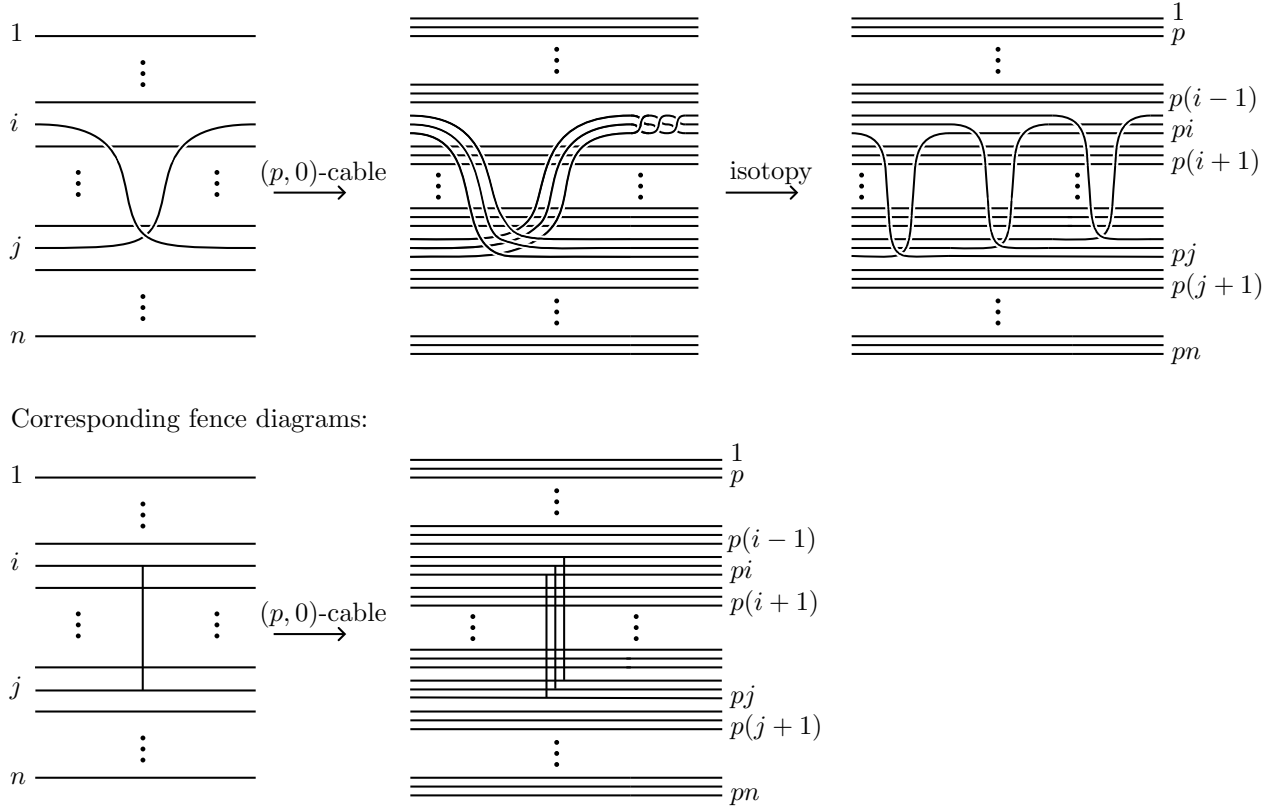 
        \caption{Top: The BKL-generator $\sigij$ and its $(p,0)$-cable for $p=3$. Bottom: Corresponding fence diagrams.}
        \label{fig:cabling_BKL-gen}
    \end{figure}

    \begin{figure}[htbp] 
        \centering
        \def\svgwidth{0.5\columnwidth}
	       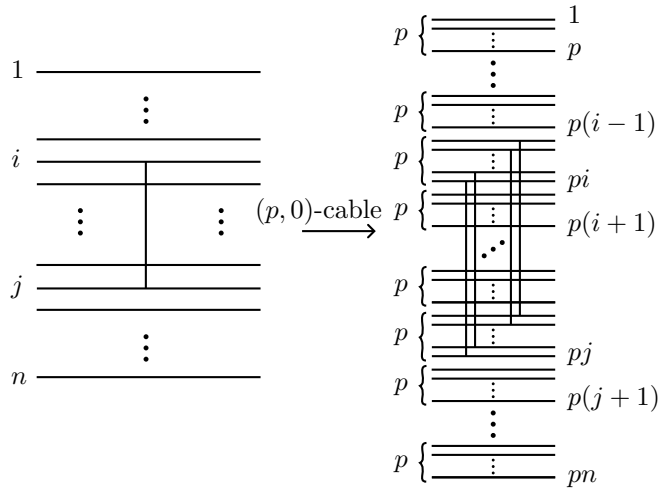 
        \caption{General fence diagram of the $(p,0)$-cable of a BKL-generator $\sigij$.}
        \label{fig:cabling_BKL-gen_fence}
    \end{figure}

    It remains to show that $B_{p,q}$ contains the dual Garside element $\delta_{pn} = \sigma_1 \sigma_2 \cdots \sigma_{pn-1}$ for every~$q \geq n$. To that end, note that, after an isotopy, the standard braid diagram $D_{p,0}$ contains the following sequence of BKL-generators, which corresponds to the ``cabled'' $\delta_n = \sigma_1 \sigma_2 \cdots \sigma_{n-1}$ in $B$:
    \begin{align}\label{eq:cabled_delta}
        \begin{split}
            \sigma_1 \sigma_2 \cdots \sigma_{pn-1} \cdot
            &\prod_{k=1}^{n-1} \left( a_{kp-1,(k+1)p-1} a_{kp-2,(k+1)p-2} \cdots a_{(k-1)p+1,kp+1} \right)\\
            \cdot &\prod_{k=1}^{n-1} \left(a_{kp-1}^{-1} \cdots a_{(k-1)p+1}^{-1} \right).
        \end{split}
    \end{align}
    Let us explain this in more detail. Again, the braid isotopy is best understood in an example; see \Cref{fig:cabling_delta} for the case $p=3$ and $n=4$.

    \begin{figure}[htbp] 
          \centering
     \def\svgwidth{1\columnwidth}
	 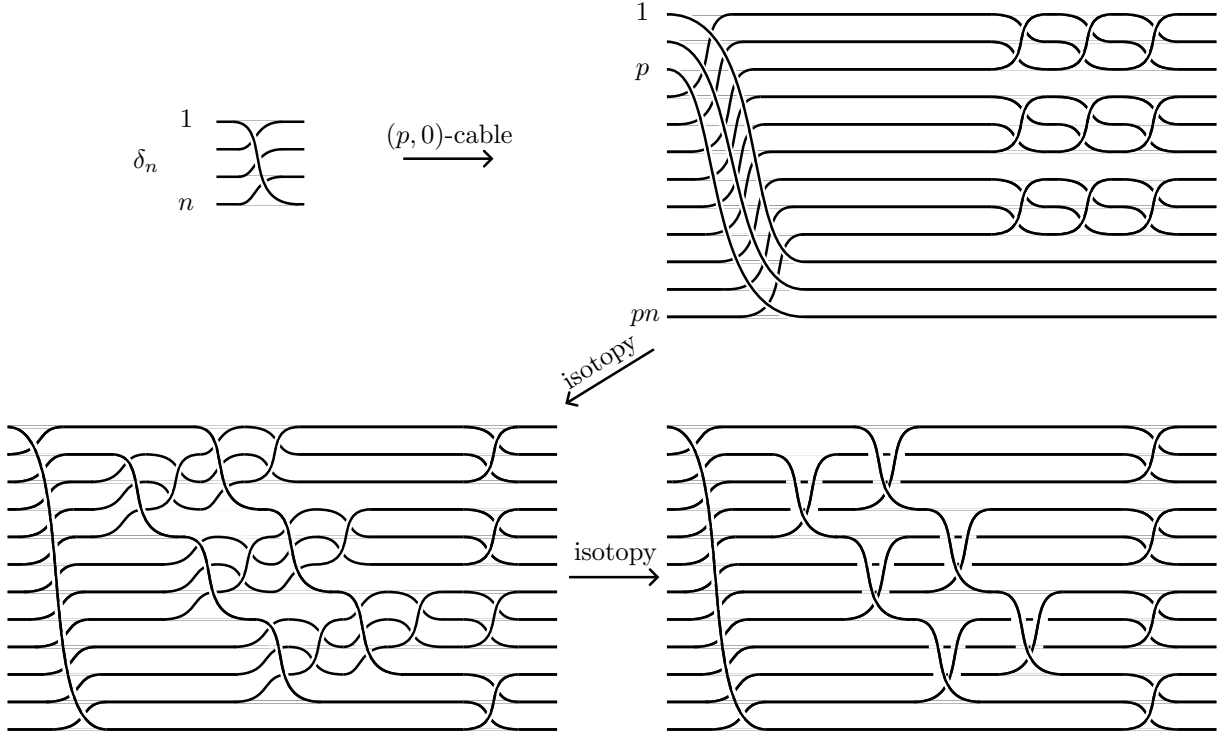 
        \caption{The $(p,0)$-cable of the dual Garside element $\delta_n = \sigma_1 \sigma_2 \cdots \sigma_{n-1} \in B_n$. Here~$n=4$ and~$p=3$. 
        }
        \label{fig:cabling_delta}
    \end{figure}
    
    For the general case, note that the first (\ie top-most) strand of $D_{p,0}$ can be ``pulled'' to the very left by a braid isotopy to make the braid word for $D_{p,0}$ start with $\delta_{pn}$; more precisely, by inserting $(p-1)$ positive and $(p-1)$ negative crossings in the $p$ parallel strands containing the top-most strand. As before, the other $(p-1)$ strands in the ``cabled'' first strand of $\delta$ each give rise to BKL-generators, and there will be in total $(n-1)\cdot (p-1)$ positive such (long) bands. Indeed, there are $(p-1)$ BKL-generators for each letter in $\delta_n$. However, note that by pulling the first strand to the very left and ``using it'' to become~$\delta_{pn}$, some of the negative $(1/p)$-twists that were inserted initially to obtain $D_{p,0}$ will not be paired with positive $(1/(p+1))$-twists as before to become positive BKL-generators. In fact, precisely~$n$ negative $(1/p)$-twists will remain, which correspond to the product $\prod_{k=1}^{n-1} \left(a_{kp-1}^{-1} \cdots a_{(k-1)p+1}^{-1} \right)$ in \eqref{eq:cabled_delta}. Since $D_{p,q}$ contains additional $q$ positive $(1/p)$-twists when built from $D_{p,0}$, these $n$ negative $(1/p)$-twists will cancel out with the $q$ positive ones whenever $q \geq n$. To summarize, the standard braid diagram $D_{p,q}$ for $K_{p,q}$ can be isotoped to a BKL-positive braid diagram on $pn$ strands which contains~$\delta_{pn}$ whenever $q \geq n$. This proves the claim.
\end{proof}



\begin{example}\label{ex:cable_converse}
The positive trefoil $T_{2,3}$ gives an example of a knot where the condition $q \geq n$ in \Cref{prop:cables} is in fact necessary. Indeed, $T_{2,3}$ is the closure of the positive braid~$\sigma_1^3$ on $n=2$ strands, which is a staircase braid. But despite being strongly quasipositive and fibered, the cable knot $(T_{2,3})_{2,1}$ is not the boundary of a positive Hopf-plumbed basket by~\cite{melvin_morton} (see also \Cref{T-pos_sqp}). Thus $(T_{2,3})_{2,1}$ is not the closure of a staircase braid by combining \Cref{theorem:Hopf_plumbed,thm:banfield}; see also~\cite[Corollary~5.3]{banfield}.
\end{example}


\section{\texorpdfstring{$\T$-positive}{T-positive} knots with small crossing number}\label{sec:crossing_nr}

This section is devoted to proving the following proposition and its consequences.

\begin{proposition}\label{prop:sqp_fib_12_cross_sec}
    All 42 strongly quasipositive, fibered knots with at most 12 crossings are $\T$-positive.
\end{proposition}

\begin{corollary}\label{cor:T-pos_12cross}
    The $\T$-positive knots with at most 12 crossings are precisely the strongly quasipositive, fibered knots with at most 12 crossings. 
\end{corollary}

\begin{proof}[{Proof of \Cref{cor:T-pos_12cross}}]
    Since $\T$-positive knots are strongly quasipositive and fibered (by definition and \Cref{theorem:Hopf_plumbed}, respectively), this directly follows from \Cref{prop:sqp_fib_12_cross_sec}. 
\end{proof}


\begin{corollary}
\label{prop:pos_fib_small_crossing_nr}
    All 33 positive, fibered knots with at most 12 crossings are $\T$-positive. 
\end{corollary}

\begin{proof}[{Proof of \Cref{prop:pos_fib_small_crossing_nr}}]
    There are $33$ knots with at most $12$ crossings which are positive and fibered; see~\cite{knotinfo}. Since positive knots are strongly quasipositive \cite{Rudolph_positiveLinksSQP,nakamura}, the corollary is a direct consequence of \Cref{prop:sqp_fib_12_cross_sec}. 
\end{proof}

\begin{proof}[{Proof of \Cref{prop:sqp_fib_12_cross_sec}}]
There are $42$ knots with at most $12$ crossings that are both strongly quasipositive and fibered~\cite{knotinfo}. $17$ of them are braid positive and thus $\T$-positive by the inclusions~\eqref{eq:inclusions} from the introduction. Five of the remaining $25$ knots have braid index~$3$. By \Cref{prop:braid_index_3} in~\Cref{sec:braid_index} below, these five knots are also $\T$-positive. Indeed, the strongly quasipositive braid representatives for these five knots with braid index $3$ are easily seen to be staircase braid representatives in $B_3$. For the remaining~$20$ knots (all of braid index~$4$), we provide either staircase braid representatives or Seifert surfaces that are positive Hopf-plumbed baskets; see \Cref{table:12cross} in \Cref{sec:table}:
\begin{itemize}
\item We found the staircase braid representatives by manipulating the strongly quasipositive braid representatives from \cite{knotinfo} using conjugations and the braid relations~\eqref{rel:1} and~\eqref{rel:2}. To avoid confusion, note that the conventions for BKL-/band generators of~\cite{knotinfo} do not match our conventions in this paper. \Cref{table:12cross} uses our conventions from \Cref{subsubsec:braids_present}.
\item We found the positive Hopf-plumbed baskets with the help of verified computations within SnapPy~\cite{SnapPy} and Sage~\cite{sagemath}. \Cref{fig:12n_148,fig:12n_329,fig:12n_366,fig:12n_402,fig:12n_528}, which are referenced in \Cref{table:12cross}, are listed in \Cref{app:figures}.
\end{itemize} 
In conclusion, all $42$ strongly quasipositive, fibered knots with at most $12$ crossings are $\T$-positive by \Cref{theorem:Hopf_plumbed,thm:banfield}. 
\end{proof}

Since staircase braids are easier to verify, we have chosen to present them rather than positive Hopf-plumbed baskets or $\T$-positive braid representatives in most cases in the proof of \Cref{prop:sqp_fib_12_cross_sec}. However, it is fairly easy to convert, for example, a staircase braid representative into a description of a positive Hopf-plumbed basket, as demonstrated in \cite[Lemma~5.5]{banfield}.

\section{Further observations on \texorpdfstring{$\T$-positivity}{T-positivity}}\label{sec:further_results}

In this section, we explore the relationship between $\T$-positivity and the unknotting number, the braid index, connected sums, visual primeness and positive trefoil plumbings. 

\subsection{\texorpdfstring{$\T$-positivity}{T-positivity} and unknotting number}

The \emph{unknotting number} $u(K)$ of a knot $K$ is the minimal number of crossing changes needed to transform a diagram of $K$ into a diagram of the unknot. The equality $u(K) = g(K)$ is known to hold true for braid positive knots $K$ and more generally, for positive trefoil plumbings $K$; see~\cite{rudolph_braidedsurfaces,BD:trefoil_plumbing,kegel2024unknottingfiberedpositiveknots}. Moreover, Stoimenow~\cite[Conjecture~4.1]{stoimenow} conjectured that $u(K)=g(K)$ is true for positive, fibered knots $K$ (see also \cite[Conjecture~5]{murasugi-Przytycki}). The authors of \cite{kegel2024unknottingfiberedpositiveknots} presented infinitely many potential counterexamples $\mathcal{K}_n$, $n \geq 1$, to this conjecture. In light of the following example, it seems worth noting that all of them are $\T$-positive knots, as explained in \Cref{T-pos_braid_pos}. 

\begin{example}
The $\T$-positive knot $K = m(12n_{642})$ fulfills $u(K) \neq g(K)$, since it has $u(K) \geq 3$, but~$g(K) \leq 2$; see also \cite{knotinfo}. In fact, from the knot diagram of~$K$ drawn in \Cref{fig:fiber_surf} (right), it is easy to find a Seifert surface of genus~$2$ for~$K$ which is a positive Hopf-plumbed basket. So $K$ is $\T$-positive by \Cref{theorem:Hopf_plumbed} and~$g(K)=2$ (see \Cref{fiber_surf}). On the other hand, a Seifert matrix computation shows that the first homology of the double branched cover of~$S^3$ branched along~$K$ with coefficients in $\F_3$ has dimension $3$ as an $\F_3$-vectorspace, so $3 \leq u(K)$~\cite{Wendt1937}.
\end{example}

\subsection{\texorpdfstring{$\T$-positivity}{T-positivity} and braid index}\label{sec:braid_index}

\Cref{question:pos_fib_T-pos} asks whether there are positive, fibered knots which are not $\T$-positive. In \Cref{sec:crossing_nr}, we have seen that there are no such knots with low crossing numbers. Another frequently used measure of complexity of knots is the \emph{braid index}, that is, the minimal number of strands of a braid representing a given knot. We now briefly examine the relation between $\T$-positivity and the braid index.

\begin{proposition}[{\cite{stoimenow:props_3-braids}}]\label{prop:braid_index_3}
    If $K$ is a strongly quasipositive, fibered knot that arises as the closure of a braid on three strands, then $K$ has a staircase braid representative on three strands, so it is $\T$-positive. 
\end{proposition}

\begin{proof}
    By Stoimenow \cite[Theorem 1.1]{stoimenow:props_3-braids}, there is a strongly quasipositive $3$-braid that represents~$K$. Since $K$ is fibered, from~\cite[Theorem~3.3, $(2) \Rightarrow (5)$]{stoimenow:props_3-braids} it follows that $K$ has a $3$-braid representative which contains $\sigma_1^k \sigma_2^\ell \sigma_1^m$ or $\sigma_2^k \sigma_1^\ell \sigma_2^m$ as subword for $k, \ell, m > 0$, or is $\sigma_1^k \sigma_2^\ell$. Thus $K$ has a staircase representative containing $\delta=\sigma_1 \sigma_2$ on three strands and is therefore $\T$-positive by \Cref{thm:banfield}.
\end{proof}

\begin{corollary}\label{cor:counterex_br_index}
    No knots with braid index $2$ or $3$ provide counterexamples to \Cref{question:pos_fib_T-pos}.
\end{corollary}

\begin{proof}
    If the closure $K$ of a braid on three strands is positive and fibered, it is strongly quasipositive~\cite{Rudolph_positiveLinksSQP,nakamura} and fibered. By \Cref{prop:braid_index_3}, $K$ is then also $\T$-positive.
\end{proof}

\Cref{prop:braid_index_3} implies that, if a knot $K$ is $\T$-positive and of braid index $3$, then it is the closure of a staircase braid on three strands. However, a $\T$-positive knot with braid index $3$ is not necessarily the closure of a $T$-positive braid on three strands for some espalier $T$, as the following example shows. 
    
\begin{example}
    The knot $K=10_{161}$ has braid index $3$ and $a_1^2 a_{1,3} a_2 a_1^2 a_2^2$ is a staircase braid representative on three strands. However, if $K$ had a $T$-positive braid representative for some espalier $T$ with three vertices, then up to conjugating this braid, it would be a positive $3$-braid representative for $K$. So $K$ would be braid positive, which it is not \cite{knotinfo}.
\end{example}

\subsection{\texorpdfstring{$\T$-positivity}{T-positivity} and connected sum}

We record the following relations of $\T$-positivity and $\T$-homogeneity to the operation of connected sum of knots.

\begin{proposition}\label{prop:conn_sum}
Let $T_1$ and $T_2$ be espaliers.
Let $K_1$ and $K_2$ be knots that are closures of a $T_1$-homogeneous braid and a $T_2$-homogeneous braid, respectively. Then the connected sum $K_1\# K_2$ is the closure of a $T_1*T_2$-homogeneous braid, where $T_1\ast T_2$ is the espalier given by the vertex-connected sum of $T_1$ and $T_2$ along the right-most vertex of $T_1$ and the left-most vertex of $T_2$.
\end{proposition}

\begin{proof}
Since $K_1$ is the closure of a $T_1$-homogeneous braid, it is the closure of a braid that is represented by a BKL-homogeneous word, say $A$, that only consists of $T_1$-generators. Likewise, $K_2$ is the closure of a braid that is represented by a BKL-homogeneous word, say $B$, that only consist of $T_2$-generators.
Let~$n_1$ be the number of vertices of $T_1$ and let~$n_2$ be the number of vertices of $T_2$. Equivalently, $n_1$ and $n_2$ are the number of strands of $A$ and $B$, respectively.  We define $i_{n_1,n_2}:B_{n_1}\to B_{n_1+n_2-1}$ to be the inclusion homomorphism that sends a BKL-generator $a_{i,j}$, $1\leq i<j\leq n_1$ to $a_{i,j}$ (interpreted as a generator of~$B_{n_1+n_2-1}$). Similarly, we define $j_{n_1,n_2}:B_{n_2}\to B_{n_1+n_2-1}$ to be the injective homomorphism that sends~$a_{i,j}$ to $a_{i+n_1-1,j+n_1-1}$ for all $1\leq i<j\leq n_2$.
Then the closure of $i_{n_1,n_2}(A)j_{n_1,n_2}(B)$ is the connected sum $K_1\# K_2$ because $i_{n_1,n_2}(A)$ only contains generators of the form $a_{i,j}^{\varepsilon}$ with $1\leq i<j\leq n_1$, $\varepsilon\in\{\pm 1\}$, and $j_{n_1,n_2}(B)$ only contains generators of the form $a_{i,j}^{\varepsilon}$ with $n_1\leq i<j\leq n_1+n_2-1$, $\varepsilon\in\{\pm 1\}$.

Observe that $T_1*T_2$ is again an espalier. Let $\mathcal{G}_1 \defeq G(T_1)$ and $\mathcal{G}_2 \defeq G(T_2)$ denote the set of $T_1$-generators and $T_2$-generators, respectively. Then the set of $T_1*T_2$-generators is precisely $i_{n_1,n_2}(\mathcal{G}_1)\cup j_{n_1,n_2}(\mathcal{G}_2)$.

The word $i_{n_1,n_2}(A)j_{n_1,n_2}(B)$ only contains $T_1*T_2$-generators (or their inverses). The homomorphisms~$i_{n_1,n_2}$ and $j_{n_1,n_2}$ map BKL-generators with positive signs to BKL-generators with positive signs and BKL-generators with negative signs to BKL-generators with negative signs. Since $A$ and~$B$ are $T_1$-homogeneous and $T_2$-homogeneous, respectively, for every generator $a\in \mathcal{G}_1$ there is a fixed sign~$\varepsilon_a\in\{\pm 1\}$ so that all instances of $a$ in $A$ carry the sign $\varepsilon_a$. Likewise, for every generator $b\in \mathcal{G}_2$ there is a fixed sign $\varepsilon_b\in\{\pm 1\}$ so that all instances of $b$ in $B$ carry the sign $\varepsilon_b$. It follows that all instances of $i_{n_1,n_2}(a)$ in $i_{n_1,n_2}(A)j_{n_1,n_2}(B)$ carry the same sign $\varepsilon_a$. Likewise, every instance of $j_{n_1,n_2}(b)$ in~$i_{n_1,n_2}(A)j_{n_1,n_2}(B)$ carries the same sign $\varepsilon_b$. Since the set of $T_1*T_2$-generators is exactly the set of images of $T_1$-generators and $T_2$-generators (under $i_{n_1,n_2}$ and $j_{n_1,n_2}$, respectively), for every $T_1*T_2$-generator all of its instances in~$i_{n_1,n_2}(A)j_{n_1,n_2}(B)$ carry the same sign. Since $K_1\# K_2$ is not a split link, it follows that $i_{n_1,n_2}(A)j_{n_1,n_2}(B)$ is a $T_1*T_2$-homogeneous braid.
\end{proof}

\begin{corollary}
Let $T_1$ and $T_2$ be espaliers.
Let $K_1$ and $K_2$ be knots that are closures of a $T_1$-positive braid and a $T_2$-positive braid, respectively. Then the connected sum $K_1\# K_2$ is the closure of a $T_1\ast T_2$-positive braid.
\end{corollary}

\begin{proof}
If in the proof of the previous proposition all signs in $A$ and $B$ are positive, then so are all the signs in $i_{n_1,n_2}(A)j_{n_1,n_2}(B)$. In particular, it is a $T_1*T_2$-positive braid word.
\end{proof}

The proof of Proposition~\ref{prop:conn_sum} discusses the connected sum operation on knots as an operation on braid words. On the level of the corresponding braided surfaces this is a Murasugi sum, where the disk corresponding to the last strand of $B_1$ is identified with the disk corresponding to the first strand of $B_2$ via a certain homeomorphism. Using different homeomorphisms to glue these same two disks produces different Murasugi sums, which were called \textit{braided Stallings plumbings} by Rudolph~\cite{Rudolph_V}. This operation was initially used by Stallings to prove that closures of homogeneous braids are fibered~\cite{stallings_1978}. See also~\cite{bode_hsueh} for a discussion of braided Stallings plumbings in the context of braided open books. 

If $A$ and $B$ are two BKL-words on $n_1$ and $n_2$ strands, respectively, then a braided Stallings plumbing of $A$ and $B$ contains the same letters as $i_{n_1,n_2}(A)j_{n_1,n_2}(B)$, only possibly in a different order. See the proof of Proposition~\ref{prop:conn_sum} for the definition of $i_{n_1,n_2}$ and $j_{n_1,n_2}$. In particular, the same argument as in the proof of Proposition~\ref{prop:conn_sum} proves the result for any braided Stallings plumbing (not only the connected sum) of any two $\mathcal{T}$-positive/$\mathcal{T}$-homogeneous braids (or, to be more precise, for any braided Stallings plumbing of any two braided surfaces that are represented by $\mathcal{T}$-positive/$\mathcal{T}$-homogeneous braid words).

However, it is not true that every Murasugi sum preserves the property of being $\mathcal{T}$-positive, as can be seen from the fact that there exist infinitely many positive Hopf plumbings (Murasugi sums of the $\mathcal{T}$-positive Hopf link) that are not positive Hopf-plumbed baskets \cite{misev} and therefore not $\mathcal{T}$-positive (see also \Cref{ex:T-pos_and_pos}).

\subsection{\texorpdfstring{$\T$-positivity}{T-positivity} and visual primeness}

Having proved that the connected sum of $\mathcal{T}$-positive links is again $\mathcal{T}$-positive, we now turn to the question whether it is trivial to detect from a given $\mathcal{T}$-positive braid diagram if its closure is a connected sum.

Let $L$ be a link in $S^3$. We say that an embedded sphere $S\subset S^3$ is a \textit{decomposition sphere} for $L$ if it intersects $L$ in exactly two points and it divides $L$ into two non-trivial parts in the following sense. The sphere $S$ divides $S^3$ into two balls $B_1$ and $B_2$ glued along their common boundary $S$. We write~$\alpha= L\cap B_1$ and $\beta=L\cap B_2$, both of which are simple arcs with exactly two distinct endpoints on $S$. Connecting these endpoints with any simple embedded arc $\gamma$ on $S$ thus produces two links $L_1=\alpha\cup \gamma$ and $L_2=\beta\cup\gamma$. Since $S$ is simply connected, the isotopy types of $L_1$ and $L_2$ do not depend on the chosen~$\gamma$. In order for $S$ to be a decomposition sphere we require neither $L_1$ nor $L_2$ to be an unknot.

If $L$ admits a decomposition sphere, we say that it is \textit{composite}. Otherwise, it is \textit{prime}.

In general, it can be difficult to see from a given link diagram $D$ if it represents a prime link or a composite link. However, if it is possible to draw in the diagram plane an embedded circle $C$ that intersects $D$ in exactly two points and none of the two resulting pieces of $D$ (one on the inside of $C$ and one on the outside of $C$) turns into a diagram of the unknot by joining its endpoints via an arc of $C$, then $D$ represents a composite link. We can interpret $C$ as the intersection of a decomposition sphere with the diagram plane. We call $C$ a \emph{decomposition circle} for $D$.

Diagrams with decomposition circles can thus be thought of as those link diagrams for which it is obvious that they represent composite links.

\begin{definition}
A link diagram $D$ of a link $L$ is said to be \emph{visually prime} if it satisfies the following property. If $L$ is not prime, then there exists a decomposition circle for $D$, which gives rise to a decomposition sphere for $L$.
\end{definition}

The study of link diagrams from which (lack of) primeness is easily detectable in this sense goes back to seminal work by Menasco \cite{menasco} and Cromwell \cite{cromwell93}, who showed, respectively, that alternating diagrams and positive braid diagrams are visually prime. The latter result was generalized by Ozawa~\cite{ozawa} to include all positive link diagrams. More recently, Feller, Lewark and Orbegozo Rodriguez proved that all homogeneous braid diagrams are visually prime \cite{feller_lewark_orbegozo}.

Since $\mathcal{T}$-positive braids are a generalization of positive braids (and, likewise, $\mathcal{T}$-homogeneous braids of homogeneous braids), it is natural to ask if $\mathcal{T}$-positive braid diagrams are also visually prime. We prove that this is not the case by providing a counterexample.

The following is a quick test to see if a diagram permits a decomposition circle. We associate to a diagram $D$ a planar graph as follows. First, forget about the signs of the crossings, so that $D$ becomes a planar graph with as many $4$-valent vertices as crossings of $D$ and edges corresponding to the arcs of~$D$ between two crossings. Let $\Gamma$ be the dual graph of this graph, \ie the vertices of $\Gamma$ correspond to the regions of the diagram $D$ (\ie connected components of $\mathbb{R}\backslash D$), and there is an edge between two vertices of $\Gamma$ whenever the corresponding regions of $D$ share an edge in the graph obtained from~$D$. We can view~$\Gamma$ as a metric graph by assigning to every edge the length 1. If there exists a decomposition circle for $D$, then there exists a non-trivial loop of length 2 in $\Gamma$. 

\begin{proposition}
The braid diagram of the $\T$-positive braid $a_{1,2}a_{2,3}a_{1,2}a_{2,3}a_{2,4}a_{2,4}a_{2,4}$ is not visually prime. 
\end{proposition}

\begin{proof}
Although the closure of the braid is the connected sum of two positive trefoils, the closed braid diagram fails the above test. The corresponding graph $\Gamma$, shown in \Cref{fig:not_vis_prime}, does not have any non-trivial loops of length two. In the figure, the vertices of the graph $\Gamma$ on the right are labelled as the regions of the diagram~$D$ on the left. Since there are no non-trivial loops of length two in $\Gamma$, the diagram~$D$ does not have a decomposition circle and is therefore not visually prime.
\end{proof}

\begin{figure}[htbp]
\centering
\small
\labellist
\pinlabel 1 at -100 1300
\pinlabel 2 at 150 1300
\pinlabel 3 at 140 680
\pinlabel 4 at 430 1570
\pinlabel 5 at 420 1330
\pinlabel 6 at 420 1050
\pinlabel 7 at 400 820
\pinlabel 8 at 400 520
\pinlabel 9 at 670 1720
\tiny
\pinlabel 10 at 680 1450
\pinlabel 11 at 670 1220
\small
\pinlabel 12 at 770 1570
\pinlabel 13 at 820 1330
\pinlabel 14 at 770 820
\pinlabel 15 at 1200 1150
\endlabellist
\includegraphics[height=5.5cm]{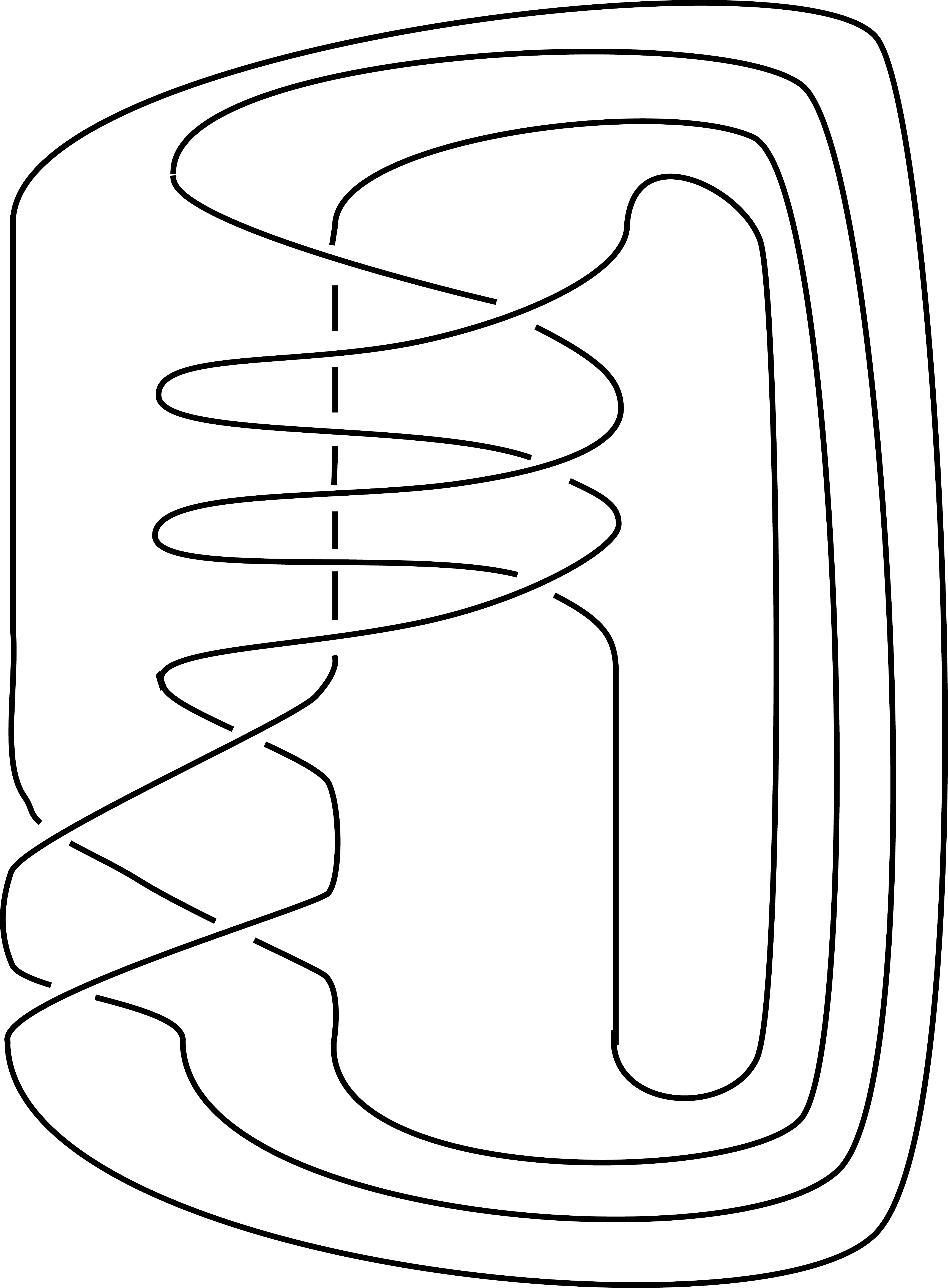}\hspace{0.5cm}
\labellist
\pinlabel 1 at 3800 1000
\pinlabel 2 at 7900 100
\pinlabel 3 at 1900 1000
\pinlabel 4 at 4600 1600
\pinlabel 5 at 6300 2400
\pinlabel 6 at 8500 4400
\pinlabel 7 at 50 130
\pinlabel 8 at 1700 2100
\pinlabel 9 at 2500 2300
\pinlabel 10 at 4700 2500
\pinlabel 11 at 7400 3800
\pinlabel 12 at 3700 2400
\pinlabel 13 at 5300 3800
\pinlabel 14 at 3600 5100
\pinlabel 15 at 3500 3300  
\endlabellist
\includegraphics[height=4.5cm]{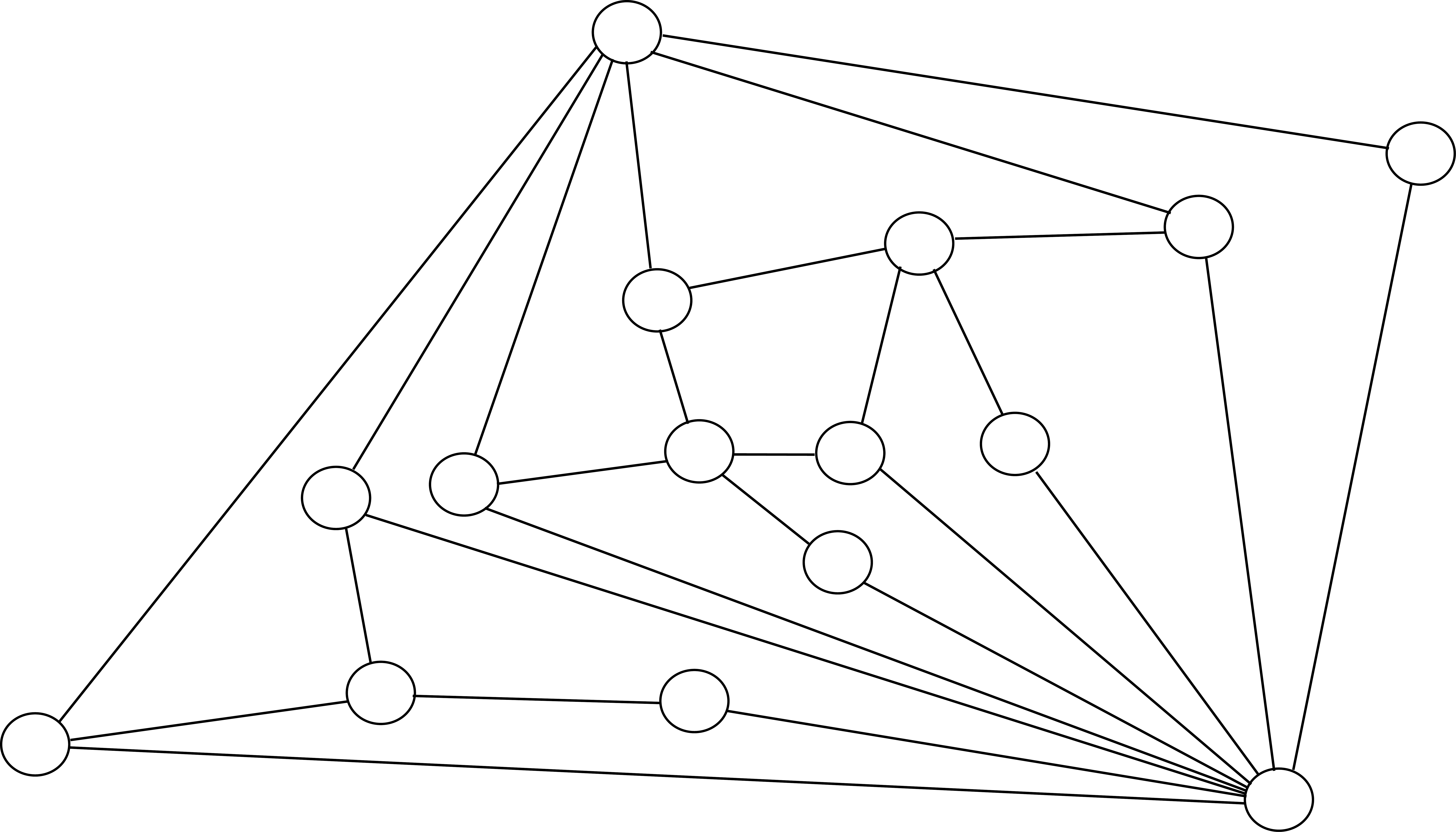}
\caption{Closed braid diagram for $T_{2,3}\# T_{2,3}$ and corresponding dual graph~$\Gamma$ with no non-trivial loops of length~$2$.}
\label{fig:not_vis_prime}
\end{figure}
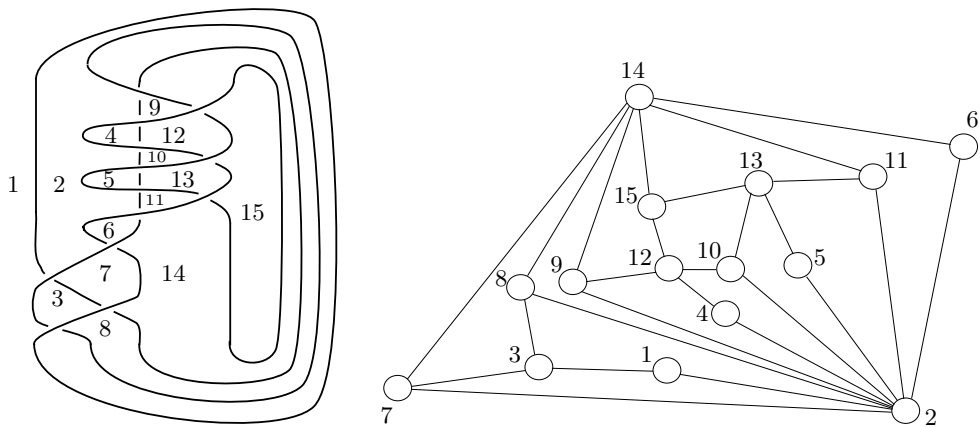

\subsection{\texorpdfstring{$\T$-positivity}{T-positivity} and positive trefoil plumbings}\label{subsec:T-pos_tref}

Let us finally compare $\T$-positivity to yet another positivity notion: positive trefoil plumbings. A fibered knot is called a \emph{positive trefoil plumbing} \cite{BD:trefoil_plumbing} if its fiber surface arises from a disk by finitely many successive plumbings of the fiber surface of the positive trefoil, \ie the torus knot~$T_{2,3}$. According to \cite[Theorem 1.4]{kegel2024unknottingfiberedpositiveknots}, the positive, fibered knot~$K\defeq 11n_{183}$ is not a positive trefoil plumbing. However, its fiber surface is a positive Hopf-plumbed basket; see \cite[Figure 5]{kegel2024unknottingfiberedpositiveknots}. So, by \Cref{theorem:Hopf_plumbed}, $K$ is an example of a $\mathcal{T}$-positive knot which is not a positive trefoil plumbing.  

\begin{question}
Are there infinitely many knots which are $\T$-positive, but not positive trefoil plumbings?
\end{question}

On the other hand, the following proposition follows almost immediately from~\cite{kegel2024unknottingfiberedpositiveknots}.

\begin{proposition}
\label{prop:inf_tref_plumb_not_T-pos}
    There exist infinitely many knots that are positive trefoil plumbings, but not $\T$-positive.
\end{proposition}

\begin{proof}
    There are infinitely many ways of plumbing two positive trefoil fiber surfaces. These are distinguished by the Alexander polynomials of their boundaries; see the proof of Proposition~4.1 in~\cite{kegel2024unknottingfiberedpositiveknots}. This gives rise to infinitely many knots of genus two that are positive trefoil plumbings. Since there are only finitely many positive Hopf-plumbed baskets of a given genus, this shows the claim.
\end{proof}

Note that we used a similar argument in \Cref{ex:T-pos_and_pos} to see that there are positive Hopf plumbings that are not $\T$-positive.
Let us also note that the knot $m(10_{145})$ is an example of a knot that is a positive trefoil plumbing, but not positive. \cite[Proof of Proposition~4.1]{kegel2024unknottingfiberedpositiveknots} provides instructions for obtaining this knot as a positive trefoil plumbing. However, we explained in \Cref{ex:T-pos_and_pos} why $m(10_{145})$ is not positive.

\newpage 
\appendix

\section{Table of \texorpdfstring{$\T$-positive}{T-positive} knots with at most 12 crossings}\label{sec:table}

\Cref{table:12cross} lists all $\T$-positive knots with at most 12 crossings, see also 
\Cref{prop:sqp_fib_12_cross_sec} and \Cref{cor:T-pos_12cross}. 

    \begin{center}
        \begin{table}[b]
            \begin{tabular}{|c c c|} 
             \hline
             Knot & Reason for $\T$-positivity & Details: staircase representative/reference \\ [0.5ex] 
             \hline \hline
             $3_1$ & braid positive & $a_1^3 \in B_2$  \\ 
             \hline
             $5_1$ & braid positive & $a_1^5 \in B_2$  \\ 
             \hline
             $7_1$ & braid positive & $a_1^7 \in B_2$  \\ 
             \hline
             $8_{19}$ & braid positive & $a_1^3 a_2 a_1^3 a_2 \in B_3$  \\ 
             \hline
             $9_1$ & braid positive & $a_1^9 \in B_2$  \\ 
             \hline
             $10_{124}$ & braid positive & $a_1^5 a_2 a_1^3 a_2 \in B_3$  \\ 
             \hline
             $10_{139}$ & braid positive & $a_1^4 a_2 a_1^3 a_2^2 \in B_3$  \\ 
             \hline
             $m(10_{145})$ & pos. Hopf-plumbed basket & \Cref{fig:fiber_surf} (left), $g=2$ \\ 
             \hline
             $10_{152}$ & braid positive &$ a_1^3 a_2^2 a_1^2 a_2^3 \in B_3$  \\ 
             \hline
             $10_{154}$ & staircase &  $a_2 a_{2,4} a_1^2 a_2 a_3 a_{2,4} a_1 a_2\in B_4 $ \\ 
             \hline
             $10_{161}$ & braid index $3$ &
             $a_1^2 a_{1,3} a_2 a_1^2 a_2^2 \in B_3$ \\ 
             \hline
             $11a_{367}$ & braid positive & $a_1^{11} \in B_2$ \\ 
             \hline
             $11n_{77}$ & braid positive & $ a_1^2 a_2^2 a_1 a_3 a_2^3 a_3^2 \in B_4$\\ 
             \hline
             $11n_{183}$ & pos. Hopf-plumbed basket  & \cite[Figure 5]{kegel2024unknottingfiberedpositiveknots} \\ 
             \hline
             $12n_{91}$ & staircase & 
             $a_3 a_1 a_{1,3} a_1 a_2 a_3^3 a_{1,3} a_2 a_3 \in B_4$
             \\ 
             \hline        
             $12n_{105}$ & staircase & 
             $a_3^3 a_2 a_{2,4} a_{1,3} a_1 a_2 a_3 a_1^2 \in B_4$
             \\ 
             \hline
             $12n_{136}$ & staircase & $a_1^2 a_2 a_3 a_2^2 a_{2,4} a_3 a_{1,3} a_2 a_3\in B_4$ \\        
             \hline
             $m(12n_{148})$ & pos. Hopf-plumbed basket  & \Cref{fig:12n_148}, $g=3$ \\
             \hline
             $12n_{187}$ & staircase & $a_3 a_2^3 a_{2,4} a_{1,3} a_1 a_2 a_3 a_1^2 \in B_4$
             \\ 
             \hline
             $12n_{242}$ & braid positive & $
             a_1 a_2^2 a_1^2 a_2^7 \in B_3$ 
             \\
             \hline
             $m(12n_{276})$ & staircase & 
             $a_{1,4} a_{2,3}^2 a_1 a_2 a_3 a_{2,4} a_{1,3}^2 \in B_4$ 
             \\
             \hline
             $12n_{328}$ & staircase & 
             $a_3 a_2 a_1 a_3 a_1 a_2^2 a_{2,4} a_1 a_2 a_3 \in B_4$
             \\ 
             \hline
             $m(12n_{329})$ & pos. Hopf-plumbed basket   & \Cref{fig:12n_329}, $g=3$ \\
             \hline
             $m(12n_{366})$ & pos. Hopf-plumbed basket  &  \Cref{fig:12n_366}, $g=3$ \\
             \hline
             $m(12n_{402})$ &  pos. Hopf-plumbed basket  & 
             
             \Cref{fig:12n_402}, $g=3$ \\
             \hline
             $12n_{417}$ & braid index $3$ & 
             $a_{1,3}^2 a_2 a_1^3 a_2^2 a_1^2 \in B_3 $
             \\ 
             \hline
             $12n_{426}$ & staircase & 
             $a_3 a_2 a_1 a_3 a_1 a_2 a_{2,4} a_1 a_2 a_3 a_2 \in B_4$
             \\ 
             \hline
             $m(12n_{472})$ & braid positive & $
             a_1 a_2^4 a_1^2 a_2^5
             \in B_3$ \\ 
             \hline
             $12n_{518}$ & staircase & $a_{2,4}^2 a_1^2 a_2 a_3 a_2 a_3^2 a_{1,3} a_2\in B_4$ \\ 
             \hline
             $m(12n_{528})$ & pos. Hopf-plumbed basket  
              &  \Cref{fig:12n_528}, $g=3$ \\
             \hline
             $12n_{574}$ & braid positive & $a_1 a_2^6 a_1^2 a_2^3\in B_3$
             \\
             \hline
             $12n_{591}$ & staircase & 
             
             $a_3 a_2 a_1 a_{1,3} a_1 a_2 a_3^2 a_2 a_1 a_{2,4}$
             \\ 
             \hline
             $12n_{640}$ & braid index $3$ & 
             $a_{1,3}^2 a_2 a_1^4 a_2^2 a_1 \in B_3$
             \\ 
             \hline
             $m(12n_{642})$ & pos. Hopf-plumbed basket  & \Cref{fig:fiber_surf} (right), $g=2$\\
             \hline
             $12n_{647}$ & braid index $3$ & 
             $a_{1,3}^2 a_2 a_1^2a_2^2a_1^3 \in B_3$
             \\ 
             \hline
             $m(12n_{660})$ & staircase & $a_2 a_{2,4}a_1^2 a_2 a_3 a_{2,4} a_1 a_{1,3} \in B_4$\\
             \hline 
             $12n_{679}$ & braid positive & $a_1^3 a_2^2 a_1^2 a_2^5\in B_3$ 
             \\
             \hline
             $12n_{688}$ & braid positive & $a_1^3 a_2^4a_1^2 a_2^3\in B_3$
             \\
             \hline
             $12n_{694}$ & staircase & $a_1 a_2 a_3^2 a_2 a_{2,4} a_3 a_{1,3} a_2 a_1 a_2 \in B_4$ \\ 
             \hline
             $12n_{725}$ & braid positive & $a_1 a_2^2a_1^4 a_2^5\in B_3$
             \\
             \hline
             $12n_{850}$ & braid index $3$ & 
             $a_{1,3}^4 a_2 a_1^2 a_2^2 a_1 \in B_3$
             \\ 
             \hline
             $12n_{888}$ & braid positive & $a_1^3 a_2^3a_1^3 a_2^3 \in B_3$
             \\
             \hline    
            \end{tabular}
            \caption{$\T$-positive knots with at most 12 crossings.}
        \label{table:12cross}
        \end{table}
    \end{center}

\section{Figures verifying \texorpdfstring{\Cref{prop:sqp_fib_12_cross_sec}}{Proposition 6.1}}\label{app:figures}

\Cref{fig:12n_148,fig:12n_329,fig:12n_366,fig:12n_402,fig:12n_528}, created using SnapPy~\cite{SnapPy}, are used used in the proof of \Cref{prop:sqp_fib_12_cross_sec}. 

\begin{figure}[htbp] 
	 \centering
     \includegraphics[width=0.58\textwidth]{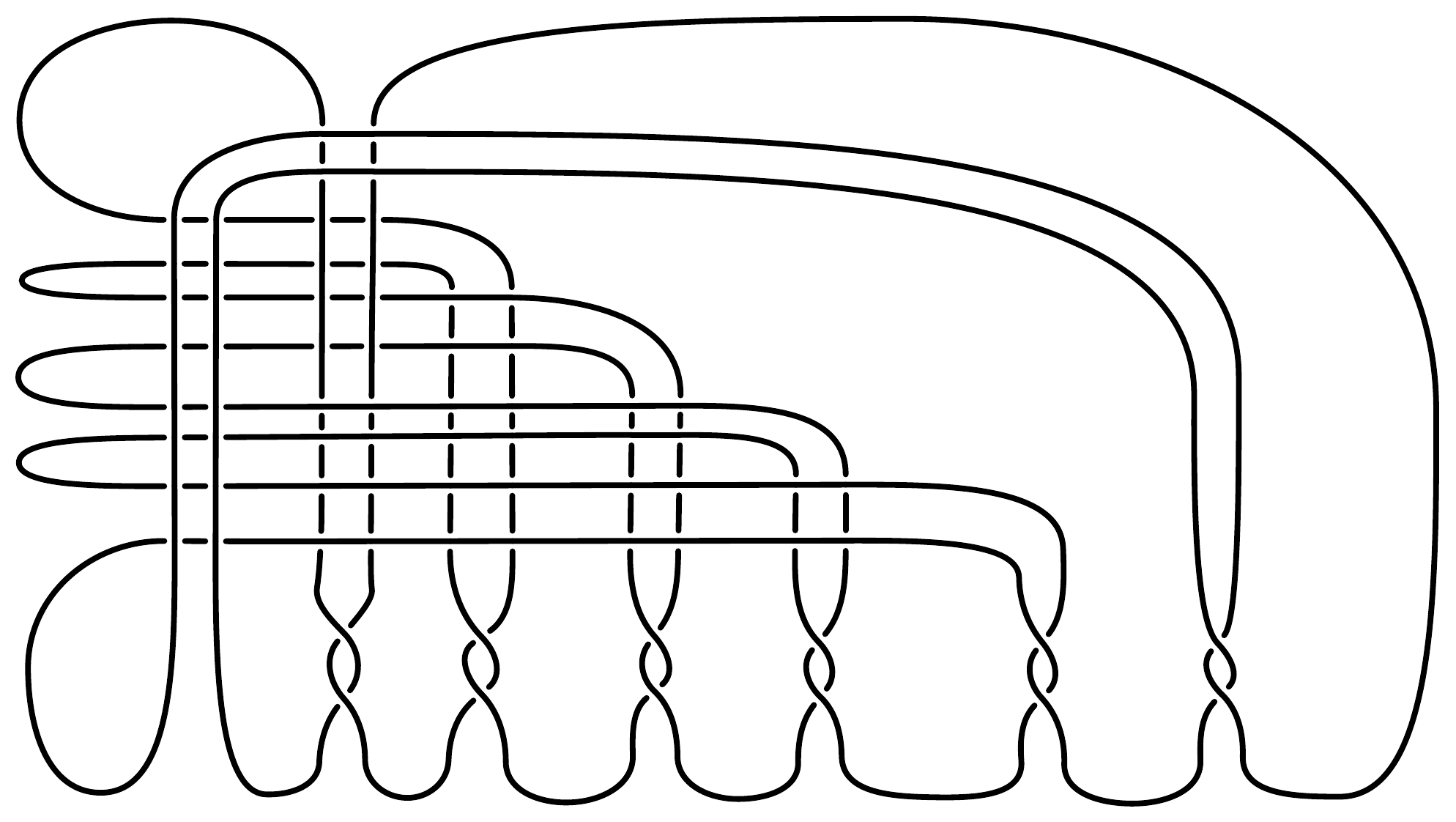}
    \caption{The knot $m(12n_{148})$ admits a genus three Seifert surface that is a positive Hopf-plumbed basket. The six Hopf bands plumbed to a disk are easily recognisable from the knot diagram (see also \Cref{fig:fiber_surf}).
    }
    \label{fig:12n_148}
\end{figure}

\begin{figure}[htbp] 
	 \centering
     \includegraphics[width=0.58\textwidth]{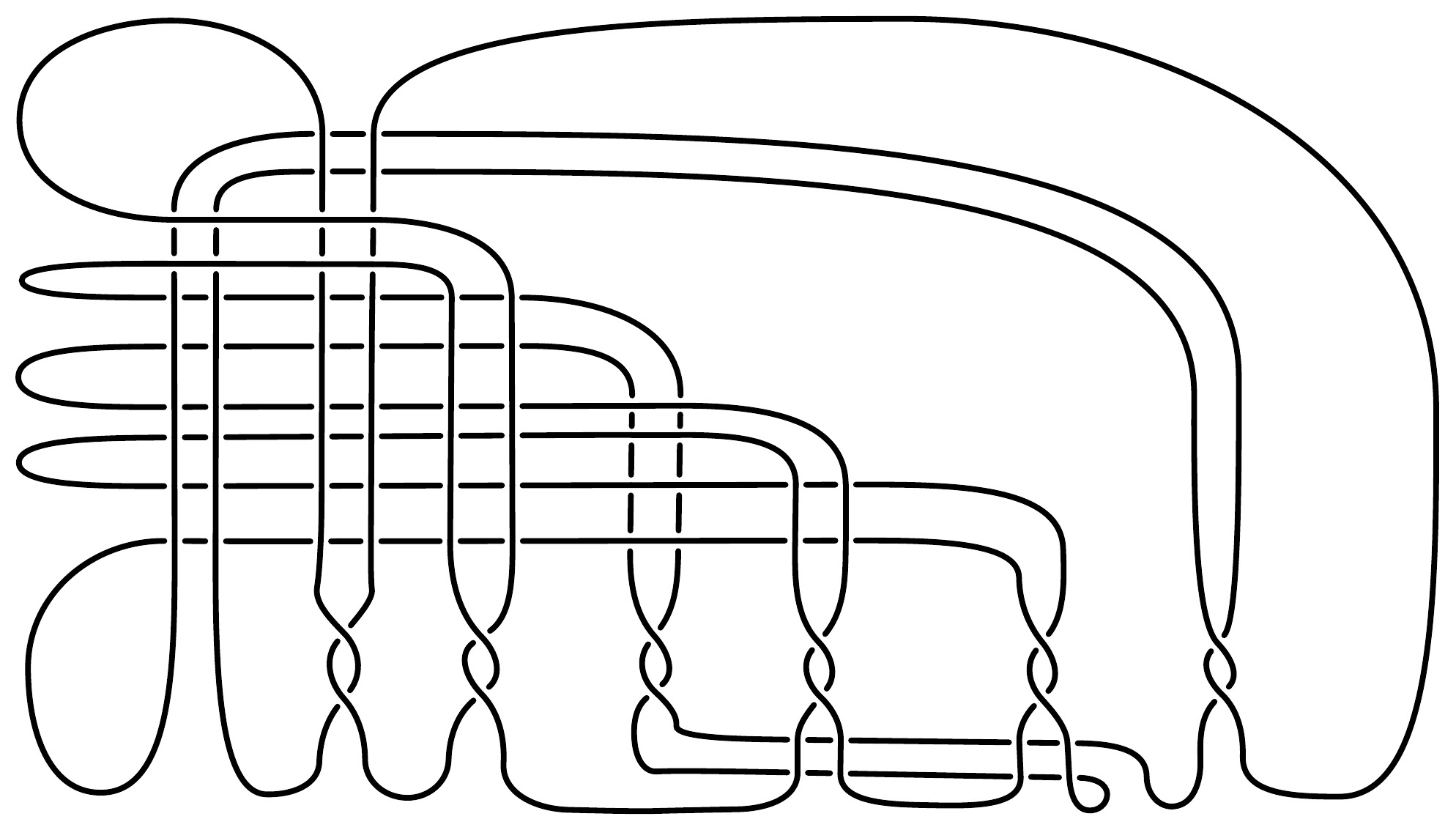}
    \caption{The knot $m(12n_{329})$ admits a genus three positive Hopf-plumbed basket.
    }
    \label{fig:12n_329}
\end{figure}

\begin{figure}[htbp] 
	 \centering
     \includegraphics[width=0.58\textwidth]{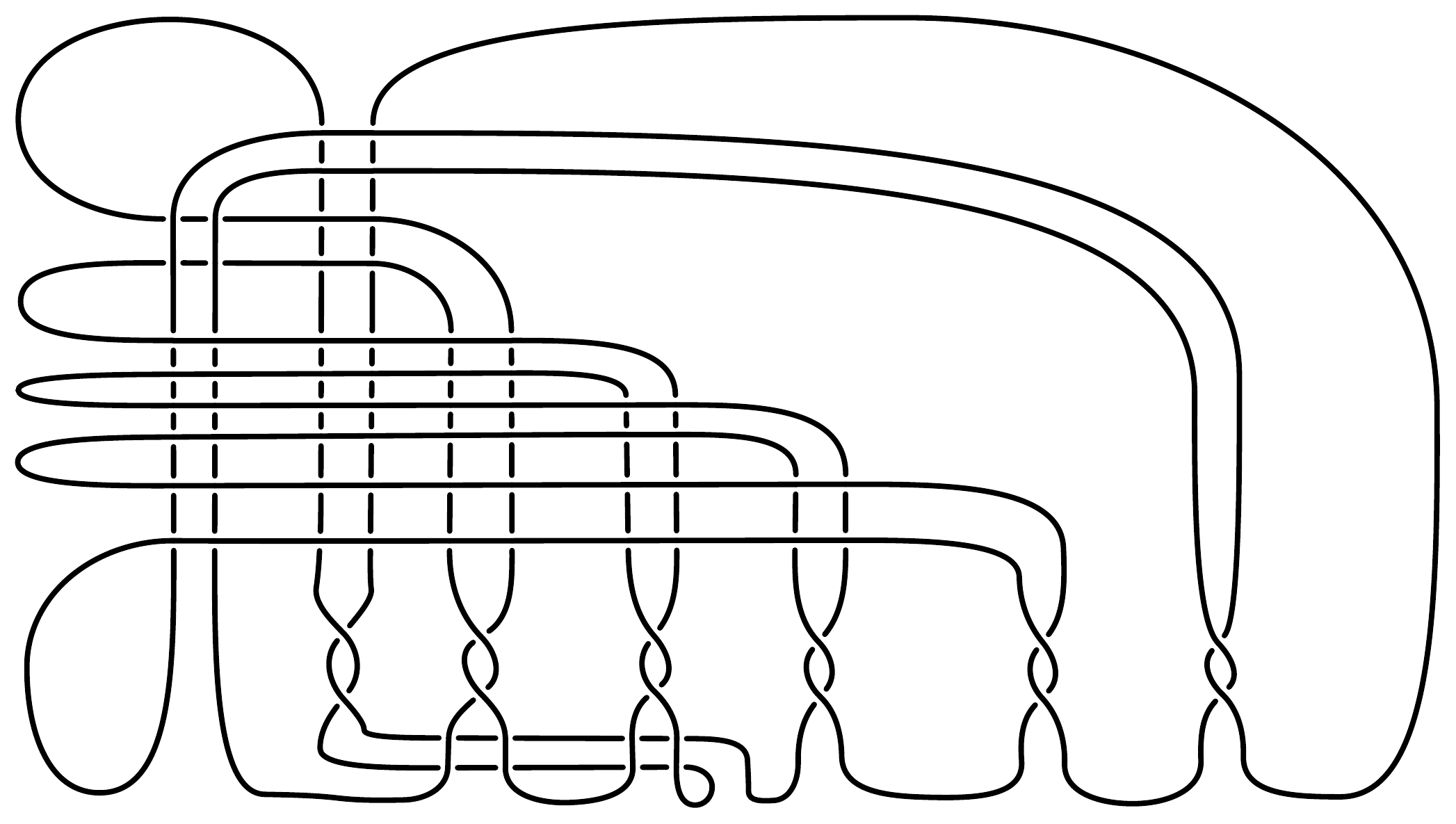}
    \caption{The knot $m(12n_{366})$ admits a genus three positive Hopf-plumbed basket.
    }
    \label{fig:12n_366}
\end{figure}

\begin{figure}[htbp] 
	 \centering
     \includegraphics[width=0.58\textwidth]{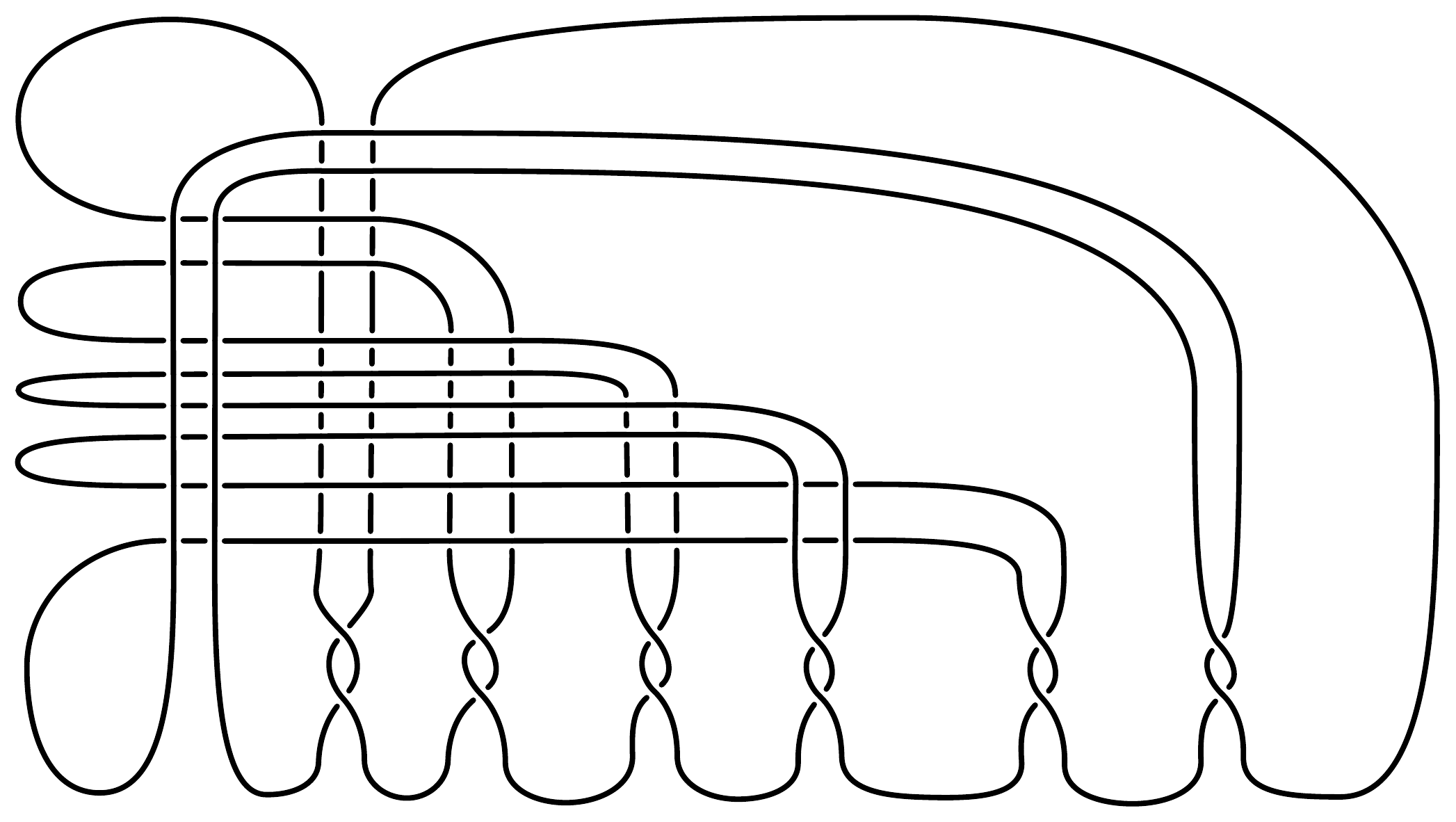}
    \caption{The knot $m(12n_{402})$ admits a genus three positive Hopf-plumbed basket.
    }
    \label{fig:12n_402}
\end{figure}

\begin{figure}[htbp] 
	 \centering
     \includegraphics[width=0.58\textwidth]{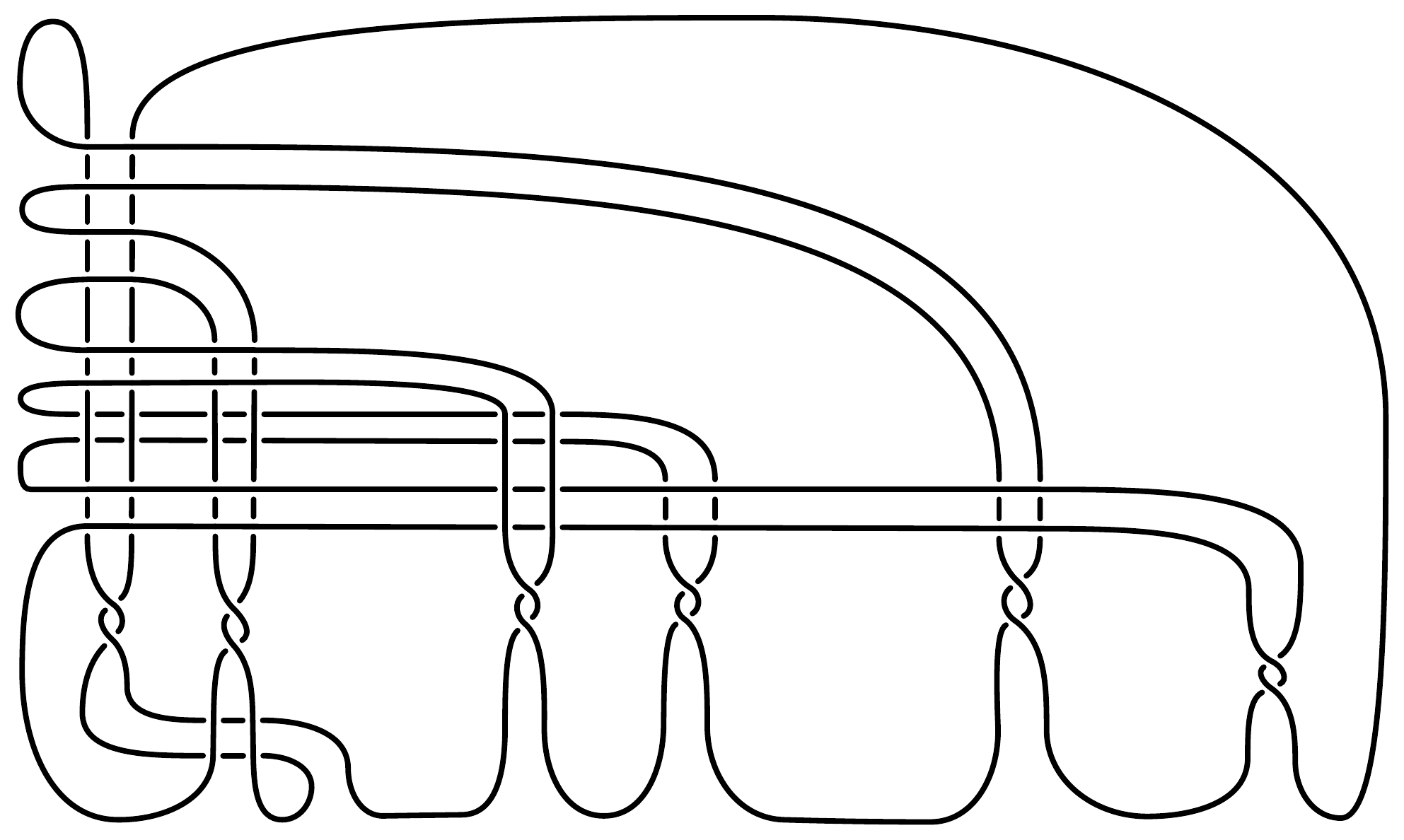}
    \caption{The knot $m(12n_{528})$ admits a genus three positive Hopf-plumbed basket.
    }
    \label{fig:12n_528}
\end{figure}

\newpage

\section{Flowcharts summarizing positivity notions}\label{subsec:flowchart}

In \Cref{fig:flowchart}, we summarize the known implications between the notions of positivity considered in this paper. \Cref{fig:flowchart2} summarizes the non-implications and emphasizes \Cref{question:pos_fib_T-pos}.

Regarding \Cref{fig:flowchart2}, note that the negative torus knots, \ie the knots $m(T_{2,{2k+1}})=T_{2,-2k-1}$, $k \geq 1$, provide easy examples of $\T$-homogeneous knots which are not $\T$-positive. These knots are the closures of the $T_2$-homogeneous braids $\sigma_1^{-2k-1}$ for $k \geq 1$. If they were $\T$-positive, they would be strongly quasipositive, thus positive by \Cref{theorem:pos_Seb}, which they are not, for example by considering their signature invariants~\cite{Traczyk}.

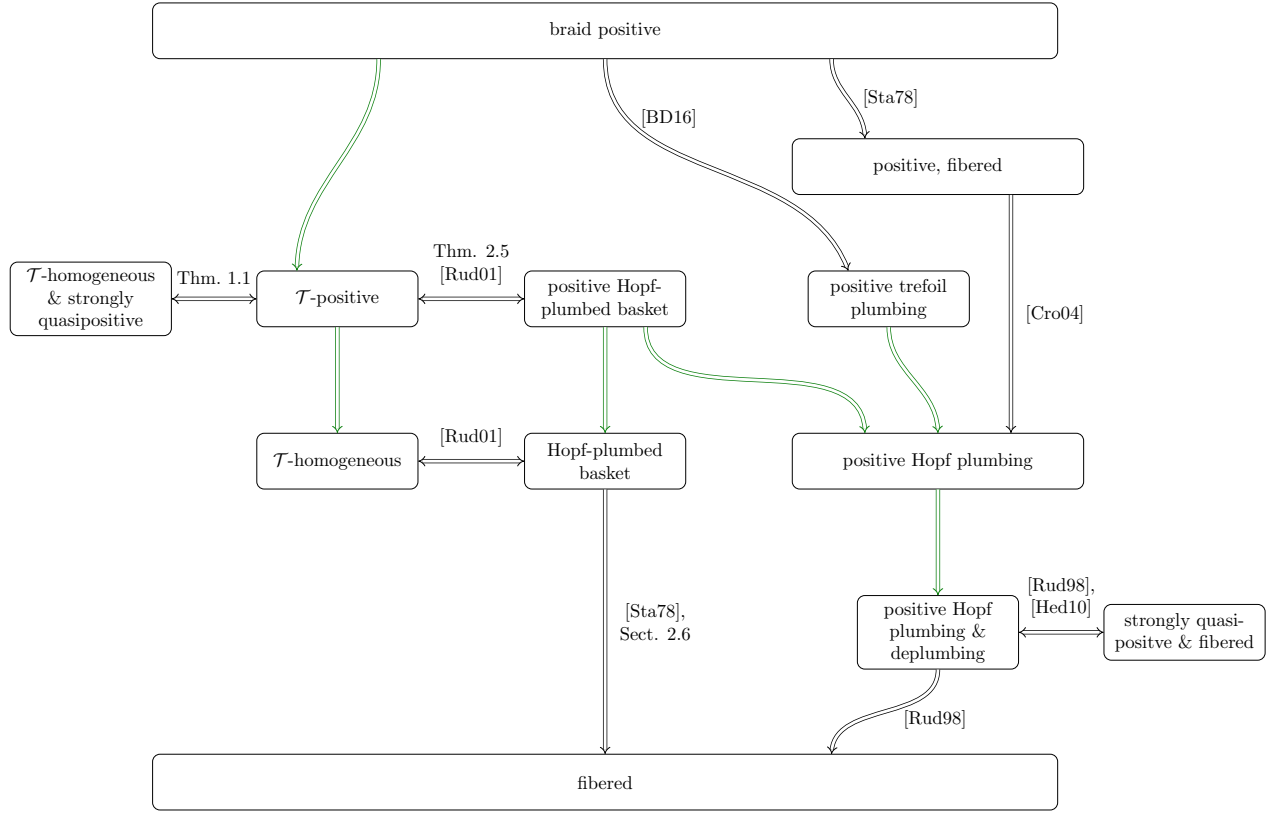
\begin{figure}[htbp]
     \centering
     \scalebox{0.7}{
 \begin{tikzpicture}[node distance=10mm and 10mm]
 \tikzstyle{box} = [rectangle, draw, text width=8em, text centered, rounded corners, minimum height=3em]
  \tikzstyle{largebox} = [rectangle, draw, text width=48em, text centered, rounded corners, minimum height=3em]
    \tikzstyle{middlebox} = [rectangle, draw, text width=15em, text centered, rounded corners, minimum height=3em]
 \tikzstyle{implies} = [-implies,double equal sign distance] 
 \tikzstyle{immediate} = [-implies,double equal sign distance, draw=Green]
 \tikzstyle{conditional} = [-implies,double equal sign distance, draw=Cerulean,dashed] 
 \tikzstyle{equiv} = [implies-implies,double equal sign distance] 

 \tikzstyle{notimplies} =[draw=red, -implies,double equal sign distance,
     preaction={draw=white,double,line width=4pt,shorten >=6pt,shorten <=6pt},
     decoration={markings, mark=at position 0.5 with {
         \draw [red, thick,-]
             ++ (-\StrikeThruDistance,-\StrikeThruDistance)
             -- ( \StrikeThruDistance, \StrikeThruDistance);
         }
     },
     postaction={decorate},
 ]
 \tikzstyle{notimplies2} =[draw=red,-implies,double equal sign distance,
     preaction={draw=white,double,line width=4pt,shorten >=6pt,shorten <=6pt},
     decoration={markings, mark=at position 0.5 with {
         \draw [red,thick,-]
             ++ (-\StrikeThruDistance,-\StrikeThruDistance)
             -- ( \StrikeThruDistance, \StrikeThruDistance);
         }
     },
     postaction={decorate}, 
 ]

 \node (bp) [largebox] {braid positive} ;
    \node (pos-basket) [box,below=40mm of bp] {positive Hopf-plumbed basket} ;
     \node (T-pos) [box,left=20mm of pos-basket] {$\T$-positive} ;
       \node (Thomo) [box,below=20mm of T-pos] {$\T$-homogeneous};
   \node (T-homo-sqp) [box,left=16mm of T-pos] {$\T$-homogeneous $\&$ strongly quasipositive};
     \node (basket) [box,below=20mm of pos-basket] {Hopf-plumbed basket} ;
 \node (tref) [box,right=23mm of pos-basket] {positive trefoil plumbing} ;     
      \node (pos-plumb) [middlebox,right=20mm of basket] {positive Hopf plumbing} ;
       \node (plumb-deplumb) [box,below=20mm of pos-plumb] {positive Hopf plumbing $\&$ deplumbing} ;
      \node (fib) [largebox,below=50mm of basket] {fibered} ;
       \node (sqp-fib) [box,right=16mm of plumb-deplumb] {strongly quasipositve $\&$ fibered};   
      \node (pos-fib) [middlebox,above=45mm of pos-plumb] {positive, fibered} ;


\draw[immediate] (bp.south south west) to[out=-90,in=90] (T-pos.north north west);
\draw[immediate] (T-pos) -- (Thomo);
\draw[immediate] (pos-basket) -- (basket);
 \draw[immediate] (pos-basket.south south east) to[out=-90,in=90] (pos-plumb.north north west);
 \draw[immediate] (pos-plumb) -- (plumb-deplumb);
  \draw[immediate] (tref.south) to[out=-90,in=90] (pos-plumb);

 \draw[implies] (bp.south south east) to[out=-90,in=90] node[pos=0.5,right,xshift=1ex] {\cite{stallings_1978}} (pos-fib.north north west);
 \draw[implies] (pos-fib.south south east) to node[midway, right,xshift=1ex] {\cite{cromwell}} (pos-plumb.north north east);
 \draw[implies] (basket) to node[pos=0.5,right,xshift=1ex,align=center] {\cite{stallings_1978},\\ Sect. \ref{Murasugi_sums}} (fib);
\draw[implies] (plumb-deplumb) to[out=-90,in=90] node[pos=0.5,right,xshift=1.5ex,yshift=-1ex,align=center] {\cite{Rudolph_V}} (fib.north north east);
\draw[implies] (bp) to[out=-90,in=110] 
node[pos=0.2,right,xshift=1ex] {\cite{BD:trefoil_plumbing}} (tref.north north west);

 \draw[equiv] (T-pos) --  node[midway,above,yshift=1ex]  {{Thm.~\ref{theorem:main_T-pos}}} (T-homo-sqp); 
 \draw[equiv] (T-pos) to node[pos=0.5,above,yshift=1ex,align=center]  {Thm. \ref{theorem:Hopf_plumbed}\\ \cite{Rudolph_Hopf_plumbing}} (pos-basket);
 \draw[equiv] (basket) -- node[midway,above,yshift=1ex] {\cite{Rudolph_Hopf_plumbing}}  (Thomo);
 \draw[equiv] (plumb-deplumb) -- node[pos=0.5,above,yshift=1ex,align=center] {\cite{Rudolph_V},\\ \cite{hedden05}} (sqp-fib);
 \end{tikzpicture}
 }
    \caption{Implications between various notions of positivity. Green arrows indicate immediate implications. Black arrows indicate non-obvious implications, for which a reference is provided.}
    \label{fig:flowchart}
\end{figure}

\newpage

\begin{figure}[htbp]
     \centering
     \scalebox{0.7}{
 \begin{tikzpicture}[node distance=10mm and 10mm]
 \tikzstyle{box} = [rectangle, draw, text width=8em, text centered, rounded corners, minimum height=3em]
  \tikzstyle{largebox} = [rectangle, draw, text width=48em, text centered, rounded corners, minimum height=3em]
    \tikzstyle{middlebox} = [rectangle, draw, text width=15em, text centered, rounded corners, minimum height=3em]
 \tikzstyle{implies} = [-implies,double equal sign distance] 
 \tikzstyle{immediate} = [-implies,double equal sign distance, draw=Green]
 \tikzstyle{conditional} = [-implies,double equal sign distance, draw=Cerulean,dashed] 
 \tikzstyle{equiv} = [implies-implies,double equal sign distance] 
  \tikzstyle{lightimplies} = [-implies,double equal sign distance, draw=gray]
   \tikzstyle{lightequiv} = [implies-implies,double equal sign distance,draw=gray] 

 \tikzstyle{notimplies} =[draw=red, -implies,double equal sign distance,
     preaction={draw=white,double,line width=4pt,shorten >=6pt,shorten <=6pt},
     decoration={markings, mark=at position 0.5 with {
         \draw [red, thick,-]
             ++ (-\StrikeThruDistance,-\StrikeThruDistance)
             -- ( \StrikeThruDistance, \StrikeThruDistance);
         }
     },
     postaction={decorate},
 ]
 \tikzstyle{notimplies2} =[draw=red,-implies,double equal sign distance,
     preaction={draw=white,double,line width=4pt,shorten >=6pt,shorten <=6pt},
     decoration={markings, mark=at position 0.5 with {
         \draw [red,thick,-]
             ++ (-\StrikeThruDistance,-\StrikeThruDistance)
             -- ( \StrikeThruDistance, \StrikeThruDistance);
         }
     },
     postaction={decorate}, 
 ]

 \node (bp) [largebox] {braid positive} ;
    \node (pos-basket) [box,below=40mm of bp] {positive Hopf-plumbed basket} ;
     \node (T-pos) [box,left=20mm of pos-basket] {$\T$-positive} ;
       \node (Thomo) [box,below=20mm of T-pos] {$\T$-homogeneous};
   \node (T-homo-sqp) [box,left=16mm of T-pos] {$\T$-homogeneous $\&$ strongly quasipositive};
     \node (basket) [box,below=20mm of pos-basket] {Hopf-plumbed basket} ;
 \node (tref) [box,right=23mm of pos-basket] {positive trefoil plumbing} ;     
      \node (pos-plumb) [middlebox,right=20mm of basket] {positive Hopf plumbing} ;
       \node (plumb-deplumb) [box,below=20mm of pos-plumb] {positive Hopf plumbing $\&$ deplumbing} ;
      \node (fib) [largebox,below=50mm of basket] {fibered} ;
       \node (sqp-fib) [box,right=16mm of plumb-deplumb] {strongly quasipositve $\&$ fibered};   
      \node (pos-fib) [middlebox,above=45mm of pos-plumb] {positive, fibered} ;

\draw[notimplies] (T-pos.north north west) to[out=90,in=-90] node[midway, left,xshift=-1ex,align=center] {$\infty$ many:\\
Prop.~\ref{prop:inf_T-pos_not_bp}}  (bp.south south west);
\draw[notimplies] (Thomo) to[out=90,in=-90] node[pos=0.5,right,xshift=1ex,align=center] {$T_{2,-2k-1}$, \\ $k \geq 1$} (T-pos);
\draw[notimplies] (basket) to[out=90,in=-90]
(pos-basket);
 \draw[notimplies] (pos-plumb.north north west) to[out=90,in=-90] node[pos=0.3,left,xshift=0ex,yshift=-3ex,align=center] {$\infty$ many: \cite{misev},\\ Sect. \ref{ex:T-pos_and_pos}}(pos-basket.south south east);
 \draw[notimplies] (plumb-deplumb) -- node[midway,right,xshift=1ex] {\cite{melvin_morton}} (pos-plumb);
  \draw[notimplies] (pos-plumb) to[out=90,in=-90] node[pos=0.75,right,xshift=1.5ex,align=center] {$11n_{183}$: \\ \cite{kegel2024unknottingfiberedpositiveknots}} (tref.south);

 \draw[notimplies] (pos-fib.north north west) to[out=90,in=-90] node[pos=0.5,right,xshift=1.5ex,align=center] {$\infty$ many: 
 \\ \cite[Prop. 5.3]{kegel2024unknottingfiberedpositiveknots}} (bp.south south east) ;
 \draw[notimplies] (pos-plumb.north north east) to node[midway, right,xshift=1ex, align=center] {$m(10_{145})$:\\ Sect. \ref{subsec:T-pos_tref}} (pos-fib.south south east) ;
 \draw[notimplies] (fib) to node[pos=0.5,right,xshift=1ex,align=center]  {\cite{melvin_morton}} (basket);
\draw[notimplies] (pos-basket.east south east) -- node[pos=0.5,below,yshift=-1ex,align=center] {$11n_{183}$: \\
\cite{kegel2024unknottingfiberedpositiveknots}} 
(tref.west south west);
\draw[notimplies] (fib.north north east) to[out=90,in=-90] node[pos=0.5,right,xshift=2.5ex,yshift=-1ex] {$4_1$} (plumb-deplumb) ;
\draw[notimplies] (pos-fib) to[out=-90,in=90,bend left=35] node[pos=0.35,left,xshift=-0.5ex,align=center] {$11n_{183}$: \\ \cite{kegel2024unknottingfiberedpositiveknots}} (tref.north north east);
\draw[notimplies] (tref.north north west) to[out=90,in=-90] node[pos=0.6,left,xshift=-1ex, align=center] {$m(10_{145})$:\\ Sect. \ref{subsec:T-pos_tref}}  (pos-fib.south south west);
\draw[notimplies] (T-pos.north north east) to[out=90,in=180] node[pos=0.55,below,yshift=-1ex,align=center] {$m(10_{145})$: \\ 
 Sect. \ref{ex:T-pos_and_pos}}  (pos-fib.west south west);

 \draw[lightequiv] (T-pos) --  node[midway,above,yshift=1ex]  {\textcolor{gray}{{Thm.~\ref{theorem:main_T-pos}}}} (T-homo-sqp); 
 \draw[lightequiv] (T-pos) to node[pos=0.5,above,yshift=1ex,align=center]  {\textcolor{gray}{Thm. \ref{theorem:Hopf_plumbed}}\\\textcolor{gray}{ \cite{Rudolph_Hopf_plumbing}}} (pos-basket);
 \draw[lightequiv] (basket) -- node[midway,above,yshift=1ex] {\textcolor{gray}{\cite{Rudolph_Hopf_plumbing}}}  (Thomo);
 \draw[lightequiv] (plumb-deplumb) -- node[pos=0.5,above,yshift=1ex,align=center] {\textcolor{gray}{\cite{Rudolph_V},}\\ \textcolor{gray}{\cite{hedden05}}} (sqp-fib);

 \draw[conditional] (pos-fib.west north west) to[out=180,in=90] node[pos=0.4,left,yshift=3ex] {\textcolor{Cerulean}{\Cref{question:pos_fib_T-pos}}} (T-pos.north);
 \draw[notimplies] (tref.west north west) -- node[pos=0.7,above,yshift=1ex,align=center] {$\infty$ many:\\ Prop.~\ref{prop:inf_tref_plumb_not_T-pos}}
 (pos-basket.east north east);
 \end{tikzpicture}
 }
   \caption{Non-implications between various notions of positivity and \Cref{question:pos_fib_T-pos}. Red arrows indicate non-implications and are accompanied by a reference. The red arrow without a reference follows from another red arrow and the gray equivalences. The arrow relating to \Cref{question:pos_fib_T-pos} is dashed blue.}
    \label{fig:flowchart2}
\end{figure}
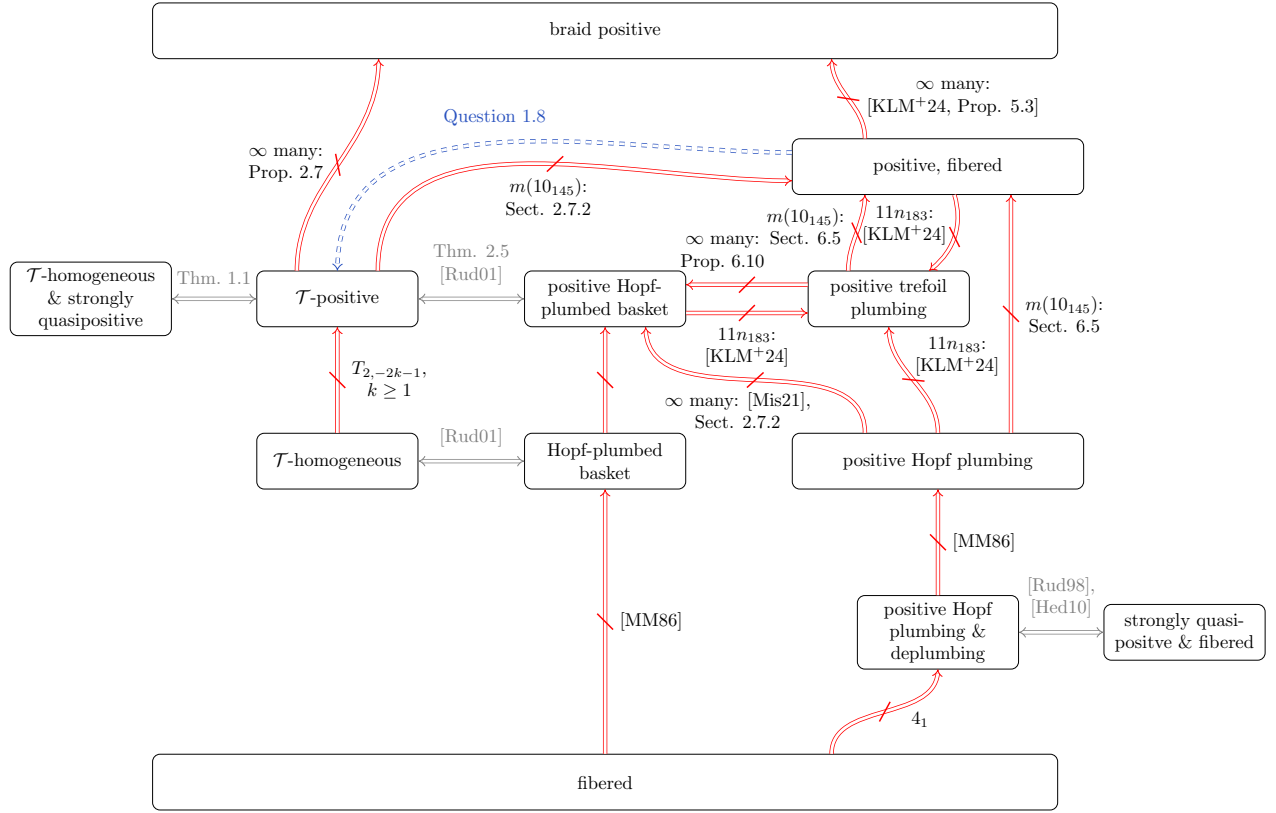

 \newpage

\bibliographystyle{alpha}
\bibliography{bibliography}

\end{document}

%% file: sqpbraidgen6.pdf_tex
\begingroup%
  \makeatletter%
  \providecommand\color[2][]{%
    \errmessage{(Inkscape) Color is used for the text in Inkscape, but the package 'color.sty' is not loaded}%
    \renewcommand\color[2][]{}%
  }%
  \providecommand\transparent[1]{%
    \errmessage{(Inkscape) Transparency is used (non-zero) for the text in Inkscape, but the package 'transparent.sty' is not loaded}%
    \renewcommand\transparent[1]{}%
  }%
  \providecommand\rotatebox[2]{#2}%
  \newcommand*\fsize{\dimexpr\f@size pt\relax}%
  \newcommand*\lineheight[1]{\fontsize{\fsize}{#1\fsize}\selectfont}%
  \ifx\svgwidth\undefined%
    \setlength{\unitlength}{157.39740327bp}%
    \ifx\svgscale\undefined%
      \relax%
    \else%
      \setlength{\unitlength}{\unitlength * \real{\svgscale}}%
    \fi%
  \else%
    \setlength{\unitlength}{\svgwidth}%
  \fi%
  \global\let\svgwidth\undefined%
  \global\let\svgscale\undefined%
  \makeatother%
  \begin{picture}(1,1.27241402)%
    \lineheight{1}%
    \setlength\tabcolsep{0pt}%
    \put(0,0){\includegraphics[width=\unitlength,page=1]{sqpbraidgen6.pdf}}%
    \put(-0.00154919,1.23691415){\color[rgb]{0,0,0}\makebox(0,0)[lt]{\lineheight{1.25}\smash{\begin{tabular}[t]{l}$1$\\\end{tabular}}}}%
    \put(-0.0032449,0.87042958){\color[rgb]{0,0,0}\makebox(0,0)[lt]{\lineheight{1.25}\smash{\begin{tabular}[t]{l}$i$\\\end{tabular}}}}%
    \put(-0.00163293,0.00813169){\color[rgb]{0,0,0}\makebox(0,0)[lt]{\lineheight{1.25}\smash{\begin{tabular}[t]{l}$n$\\\end{tabular}}}}%
    \put(-0.00157012,0.37136626){\color[rgb]{0,0,0}\makebox(0,0)[lt]{\lineheight{1.25}\smash{\begin{tabular}[t]{l}$j$\end{tabular}}}}%
  \end{picture}%
\endgroup%

%% file: lineargraphT_n.pdf_tex
\begingroup%
  \makeatletter%
  \providecommand\color[2][]{%
    \errmessage{(Inkscape) Color is used for the text in Inkscape, but the package 'color.sty' is not loaded}%
    \renewcommand\color[2][]{}%
  }%
  \providecommand\transparent[1]{%
    \errmessage{(Inkscape) Transparency is used (non-zero) for the text in Inkscape, but the package 'transparent.sty' is not loaded}%
    \renewcommand\transparent[1]{}%
  }%
  \providecommand\rotatebox[2]{#2}%
  \newcommand*\fsize{\dimexpr\f@size pt\relax}%
  \newcommand*\lineheight[1]{\fontsize{\fsize}{#1\fsize}\selectfont}%
  \ifx\svgwidth\undefined%
    \setlength{\unitlength}{189.32101272bp}%
    \ifx\svgscale\undefined%
      \relax%
    \else%
      \setlength{\unitlength}{\unitlength * \real{\svgscale}}%
    \fi%
  \else%
    \setlength{\unitlength}{\svgwidth}%
  \fi%
  \global\let\svgwidth\undefined%
  \global\let\svgscale\undefined%
  \makeatother%
  \begin{picture}(1,0.16070158)%
    \lineheight{1}%
    \setlength\tabcolsep{0pt}%
    \put(0,0){\includegraphics[width=\unitlength,page=1]{lineargraphT_n.pdf}}%
    \put(-0.00269773,0.13118775){\color[rgb]{0,0,0}\makebox(0,0)[lt]{\lineheight{1.25}\smash{\begin{tabular}[t]{l}$1$\\\end{tabular}}}}%
    \put(0.14702915,0.13118775){\color[rgb]{0,0,0}\makebox(0,0)[lt]{\lineheight{1.25}\smash{\begin{tabular}[t]{l}$2$\\\end{tabular}}}}%
    \put(0.29675933,0.13118775){\color[rgb]{0,0,0}\makebox(0,0)[lt]{\lineheight{1.25}\smash{\begin{tabular}[t]{l}$3$\\\end{tabular}}}}%
    \put(0.66658166,0.13118775){\color[rgb]{0,0,0}\makebox(0,0)[lt]{\lineheight{1.25}\smash{\begin{tabular}[t]{l}$n-1$\\\end{tabular}}}}%
    \put(0.93459296,0.13118775){\color[rgb]{0,0,0}\makebox(0,0)[lt]{\lineheight{1.25}\smash{\begin{tabular}[t]{l}$n$\\\end{tabular}}}}%
  \end{picture}%
\endgroup%

%% file: T-homo-braid-and-surface.pdf_tex
\begingroup%
  \makeatletter%
  \providecommand\color[2][]{%
    \errmessage{(Inkscape) Color is used for the text in Inkscape, but the package 'color.sty' is not loaded}%
    \renewcommand\color[2][]{}%
  }%
  \providecommand\transparent[1]{%
    \errmessage{(Inkscape) Transparency is used (non-zero) for the text in Inkscape, but the package 'transparent.sty' is not loaded}%
    \renewcommand\transparent[1]{}%
  }%
  \providecommand\rotatebox[2]{#2}%
  \newcommand*\fsize{\dimexpr\f@size pt\relax}%
  \newcommand*\lineheight[1]{\fontsize{\fsize}{#1\fsize}\selectfont}%
  \ifx\svgwidth\undefined%
    \setlength{\unitlength}{539.51474167bp}%
    \ifx\svgscale\undefined%
      \relax%
    \else%
      \setlength{\unitlength}{\unitlength * \real{\svgscale}}%
    \fi%
  \else%
    \setlength{\unitlength}{\svgwidth}%
  \fi%
  \global\let\svgwidth\undefined%
  \global\let\svgscale\undefined%
  \makeatother%
  \begin{picture}(1,0.52822212)%
    \lineheight{1}%
    \setlength\tabcolsep{0pt}%
    \put(0,0){\includegraphics[width=\unitlength,page=1]{T-homo-braid-and-surface.pdf}}%
  \end{picture}%
\endgroup%

%% file: sample_graphs.pdf_tex
\begingroup%
  \makeatletter%
  \providecommand\color[2][]{%
    \errmessage{(Inkscape) Color is used for the text in Inkscape, but the package 'color.sty' is not loaded}%
    \renewcommand\color[2][]{}%
  }%
  \providecommand\transparent[1]{%
    \errmessage{(Inkscape) Transparency is used (non-zero) for the text in Inkscape, but the package 'transparent.sty' is not loaded}%
    \renewcommand\transparent[1]{}%
  }%
  \providecommand\rotatebox[2]{#2}%
  \newcommand*\fsize{\dimexpr\f@size pt\relax}%
  \newcommand*\lineheight[1]{\fontsize{\fsize}{#1\fsize}\selectfont}%
  \ifx\svgwidth\undefined%
    \setlength{\unitlength}{553.30868218bp}%
    \ifx\svgscale\undefined%
      \relax%
    \else%
      \setlength{\unitlength}{\unitlength * \real{\svgscale}}%
    \fi%
  \else%
    \setlength{\unitlength}{\svgwidth}%
  \fi%
  \global\let\svgwidth\undefined%
  \global\let\svgscale\undefined%
  \makeatother%
  \begin{picture}(1,0.08059998)%
    \lineheight{1}%
    \setlength\tabcolsep{0pt}%
    \put(0,0){\includegraphics[width=\unitlength,page=1]{sample_graphs.pdf}}%
    \put(-0.00092306,0.07050148){\color[rgb]{0,0,0}\makebox(0,0)[lt]{\lineheight{1.25}\smash{\begin{tabular}[t]{l}$1$\\\end{tabular}}}}%
    \put(0.05030776,0.07050148){\color[rgb]{0,0,0}\makebox(0,0)[lt]{\lineheight{1.25}\smash{\begin{tabular}[t]{l}$2$\\\end{tabular}}}}%
    \put(0.10153969,0.07050148){\color[rgb]{0,0,0}\makebox(0,0)[lt]{\lineheight{1.25}\smash{\begin{tabular}[t]{l}$3$\\\end{tabular}}}}%
    \put(0.15277051,0.07050148){\color[rgb]{0,0,0}\makebox(0,0)[lt]{\lineheight{1.25}\smash{\begin{tabular}[t]{l}$4$\\\end{tabular}}}}%
    \put(0.2040013,0.07050148){\color[rgb]{0,0,0}\makebox(0,0)[lt]{\lineheight{1.25}\smash{\begin{tabular}[t]{l}$5$\\\end{tabular}}}}%
    \put(0.25523214,0.07050148){\color[rgb]{0,0,0}\makebox(0,0)[lt]{\lineheight{1.25}\smash{\begin{tabular}[t]{l}$6$\\\end{tabular}}}}%
    \put(0.30646297,0.07050148){\color[rgb]{0,0,0}\makebox(0,0)[lt]{\lineheight{1.25}\smash{\begin{tabular}[t]{l}$7$\\\end{tabular}}}}%
    \put(0,0){\includegraphics[width=\unitlength,page=2]{sample_graphs.pdf}}%
    \put(0.38586963,0.07050148){\color[rgb]{0,0,0}\makebox(0,0)[lt]{\lineheight{1.25}\smash{\begin{tabular}[t]{l}$1$\\\end{tabular}}}}%
    \put(0.43710035,0.07050148){\color[rgb]{0,0,0}\makebox(0,0)[lt]{\lineheight{1.25}\smash{\begin{tabular}[t]{l}$2$\\\end{tabular}}}}%
    \put(0.48833228,0.07050148){\color[rgb]{0,0,0}\makebox(0,0)[lt]{\lineheight{1.25}\smash{\begin{tabular}[t]{l}$3$\\\end{tabular}}}}%
    \put(0.53956318,0.07050148){\color[rgb]{0,0,0}\makebox(0,0)[lt]{\lineheight{1.25}\smash{\begin{tabular}[t]{l}$4$\\\end{tabular}}}}%
    \put(0.59079394,0.07050148){\color[rgb]{0,0,0}\makebox(0,0)[lt]{\lineheight{1.25}\smash{\begin{tabular}[t]{l}$5$\\\end{tabular}}}}%
    \put(0.64202485,0.07050148){\color[rgb]{0,0,0}\makebox(0,0)[lt]{\lineheight{1.25}\smash{\begin{tabular}[t]{l}$6$\\\end{tabular}}}}%
    \put(0.69325568,0.07050148){\color[rgb]{0,0,0}\makebox(0,0)[lt]{\lineheight{1.25}\smash{\begin{tabular}[t]{l}$7$\\\end{tabular}}}}%
    \put(0,0){\includegraphics[width=\unitlength,page=3]{sample_graphs.pdf}}%
    \put(0.77266016,0.07050148){\color[rgb]{0,0,0}\makebox(0,0)[lt]{\lineheight{1.25}\smash{\begin{tabular}[t]{l}$1$\\\end{tabular}}}}%
    \put(0.82389099,0.07050148){\color[rgb]{0,0,0}\makebox(0,0)[lt]{\lineheight{1.25}\smash{\begin{tabular}[t]{l}$2$\\\end{tabular}}}}%
    \put(0.87512288,0.07050148){\color[rgb]{0,0,0}\makebox(0,0)[lt]{\lineheight{1.25}\smash{\begin{tabular}[t]{l}$3$\\\end{tabular}}}}%
    \put(0.92635371,0.07050148){\color[rgb]{0,0,0}\makebox(0,0)[lt]{\lineheight{1.25}\smash{\begin{tabular}[t]{l}$4$\\\end{tabular}}}}%
    \put(0.97758447,0.07050148){\color[rgb]{0,0,0}\makebox(0,0)[lt]{\lineheight{1.25}\smash{\begin{tabular}[t]{l}$5$\\\end{tabular}}}}%
  \end{picture}%
\endgroup%

%% file: murasugi_sum_homo5.pdf_tex
\begingroup%
  \makeatletter%
  \providecommand\color[2][]{%
    \errmessage{(Inkscape) Color is used for the text in Inkscape, but the package 'color.sty' is not loaded}%
    \renewcommand\color[2][]{}%
  }%
  \providecommand\transparent[1]{%
    \errmessage{(Inkscape) Transparency is used (non-zero) for the text in Inkscape, but the package 'transparent.sty' is not loaded}%
    \renewcommand\transparent[1]{}%
  }%
  \providecommand\rotatebox[2]{#2}%
  \newcommand*\fsize{\dimexpr\f@size pt\relax}%
  \newcommand*\lineheight[1]{\fontsize{\fsize}{#1\fsize}\selectfont}%
  \ifx\svgwidth\undefined%
    \setlength{\unitlength}{559.56959161bp}%
    \ifx\svgscale\undefined%
      \relax%
    \else%
      \setlength{\unitlength}{\unitlength * \real{\svgscale}}%
    \fi%
  \else%
    \setlength{\unitlength}{\svgwidth}%
  \fi%
  \global\let\svgwidth\undefined%
  \global\let\svgscale\undefined%
  \makeatother%
  \begin{picture}(1,0.38244252)%
    \lineheight{1}%
    \setlength\tabcolsep{0pt}%
    \put(0,0){\includegraphics[width=\unitlength,page=1]{murasugi_sum_homo5.pdf}}%
    \put(0.0954028,0.10829085){\color[rgb]{0,0,1}\makebox(0,0)[lt]{\lineheight{1.25}\smash{\begin{tabular}[t]{l}$P_1$\end{tabular}}}}%
    \put(0.0954028,0.10829085){\color[rgb]{0,0,1}\makebox(0,0)[lt]{\lineheight{1.25}\smash{\begin{tabular}[t]{l}$P_1$\end{tabular}}}}%
    \put(0.65219535,0.21074589){\color[rgb]{0,0.50196078,0}\makebox(0,0)[lt]{\lineheight{1.25}\smash{\begin{tabular}[t]{l}$P_3$\end{tabular}}}}%
    \put(0.65219535,0.16009045){\color[rgb]{0,0.50196078,0}\makebox(0,0)[lt]{\lineheight{1.25}\smash{\begin{tabular}[t]{l}$P_3$\end{tabular}}}}%
    \put(0.96068346,0.00718766){\color[rgb]{0,0,0}\makebox(0,0)[lt]{\lineheight{1.25}\smash{\begin{tabular}[t]{l}$F_1$\end{tabular}}}}%
    \put(0.96068516,0.1085038){\color[rgb]{0.00392157,0,0}\makebox(0,0)[lt]{\lineheight{1.25}\smash{\begin{tabular}[t]{l}$F_2$\end{tabular}}}}%
    \put(0.96068348,0.31113405){\color[rgb]{0,0,0}\makebox(0,0)[lt]{\lineheight{1.25}\smash{\begin{tabular}[t]{l}$F_3$\end{tabular}}}}%
    \put(0.96068348,0.20981884){\color[rgb]{0,0,0}\makebox(0,0)[lt]{\lineheight{1.25}\smash{\begin{tabular}[t]{l}$F_4$\end{tabular}}}}%
    \put(0,0){\includegraphics[width=\unitlength,page=2]{murasugi_sum_homo5.pdf}}%
    \put(0.23020485,0.36342336){\color[rgb]{0.8,0.33333333,0}\makebox(0,0)[lt]{\lineheight{1.25}\smash{\begin{tabular}[t]{l}$P_2$\end{tabular}}}}%
    \put(0.23020485,0.31276575){\color[rgb]{0.8,0.33333333,0}\makebox(0,0)[lt]{\lineheight{1.25}\smash{\begin{tabular}[t]{l}$P_2$\end{tabular}}}}%
    \put(0,0){\includegraphics[width=\unitlength,page=3]{murasugi_sum_homo5.pdf}}%
    \put(0.0954028,0.10829085){\color[rgb]{0,0,1}\makebox(0,0)[lt]{\lineheight{1.25}\smash{\begin{tabular}[t]{l}$P_1$\end{tabular}}}}%
    \put(0.0954028,0.05763324){\color[rgb]{0,0,1}\makebox(0,0)[lt]{\lineheight{1.25}\smash{\begin{tabular}[t]{l}$P_1$\end{tabular}}}}%
  \end{picture}%
\endgroup%

%% file: isotopy_neg_gen.pdf_tex
\begingroup%
  \makeatletter%
  \providecommand\color[2][]{%
    \errmessage{(Inkscape) Color is used for the text in Inkscape, but the package 'color.sty' is not loaded}%
    \renewcommand\color[2][]{}%
  }%
  \providecommand\transparent[1]{%
    \errmessage{(Inkscape) Transparency is used (non-zero) for the text in Inkscape, but the package 'transparent.sty' is not loaded}%
    \renewcommand\transparent[1]{}%
  }%
  \providecommand\rotatebox[2]{#2}%
  \newcommand*\fsize{\dimexpr\f@size pt\relax}%
  \newcommand*\lineheight[1]{\fontsize{\fsize}{#1\fsize}\selectfont}%
  \ifx\svgwidth\undefined%
    \setlength{\unitlength}{539.51474167bp}%
    \ifx\svgscale\undefined%
      \relax%
    \else%
      \setlength{\unitlength}{\unitlength * \real{\svgscale}}%
    \fi%
  \else%
    \setlength{\unitlength}{\svgwidth}%
  \fi%
  \global\let\svgwidth\undefined%
  \global\let\svgscale\undefined%
  \makeatother%
  \begin{picture}(1,0.23897016)%
    \lineheight{1}%
    \setlength\tabcolsep{0pt}%
    \put(0,0){\includegraphics[width=\unitlength,page=1]{isotopy_neg_gen.pdf}}%
    \put(0.49554671,0.03103237){\color[rgb]{0,0.50196078,0}\makebox(0,0)[lt]{\lineheight{1.25}\smash{\begin{tabular}[t]{l}$2\pi$\end{tabular}}}}%
    \put(0,0){\includegraphics[width=\unitlength,page=2]{isotopy_neg_gen.pdf}}%
  \end{picture}%
\endgroup%

%% file: cabling_BKL-gen_7.pdf_tex
\begingroup%
  \makeatletter%
  \providecommand\color[2][]{%
    \errmessage{(Inkscape) Color is used for the text in Inkscape, but the package 'color.sty' is not loaded}%
    \renewcommand\color[2][]{}%
  }%
  \providecommand\transparent[1]{%
    \errmessage{(Inkscape) Transparency is used (non-zero) for the text in Inkscape, but the package 'transparent.sty' is not loaded}%
    \renewcommand\transparent[1]{}%
  }%
  \providecommand\rotatebox[2]{#2}%
  \newcommand*\fsize{\dimexpr\f@size pt\relax}%
  \newcommand*\lineheight[1]{\fontsize{\fsize}{#1\fsize}\selectfont}%
  \ifx\svgwidth\undefined%
    \setlength{\unitlength}{776.34775237bp}%
    \ifx\svgscale\undefined%
      \relax%
    \else%
      \setlength{\unitlength}{\unitlength * \real{\svgscale}}%
    \fi%
  \else%
    \setlength{\unitlength}{\svgwidth}%
  \fi%
  \global\let\svgwidth\undefined%
  \global\let\svgscale\undefined%
  \makeatother%
  \begin{picture}(1,0.65487904)%
    \lineheight{1}%
    \setlength\tabcolsep{0pt}%
    \put(0,0){\includegraphics[width=\unitlength,page=1]{cabling_BKL-gen_7.pdf}}%
    \put(-0.00031409,0.62983315){\color[rgb]{0,0,0}\makebox(0,0)[lt]{\lineheight{1.25}\smash{\begin{tabular}[t]{l}$1$\\\end{tabular}}}}%
    \put(-0.00065788,0.55553175){\color[rgb]{0,0,0}\makebox(0,0)[lt]{\lineheight{1.25}\smash{\begin{tabular}[t]{l}$i$\\\end{tabular}}}}%
    \put(-0.00033106,0.38070875){\color[rgb]{0,0,0}\makebox(0,0)[lt]{\lineheight{1.25}\smash{\begin{tabular}[t]{l}$n$\\\end{tabular}}}}%
    \put(-0.00031833,0.45435123){\color[rgb]{0,0,0}\makebox(0,0)[lt]{\lineheight{1.25}\smash{\begin{tabular}[t]{l}$j$\end{tabular}}}}%
    \put(0,0){\includegraphics[width=\unitlength,page=2]{cabling_BKL-gen_7.pdf}}%
    \put(0.20597729,0.51548279){\color[rgb]{0,0,0}\makebox(0,0)[lt]{\lineheight{1.25}\smash{\begin{tabular}[t]{l}$(p,0)$-cable\\\end{tabular}}}}%
    \put(0,0){\includegraphics[width=\unitlength,page=3]{cabling_BKL-gen_7.pdf}}%
    \put(0.59496632,0.51583353){\color[rgb]{0,0,0}\makebox(0,0)[lt]{\lineheight{1.25}\smash{\begin{tabular}[t]{l}isotopy\end{tabular}}}}%
    \put(0,0){\includegraphics[width=\unitlength,page=4]{cabling_BKL-gen_7.pdf}}%
    \put(-0.00031409,0.2647062){\color[rgb]{0,0,0}\makebox(0,0)[lt]{\lineheight{1.25}\smash{\begin{tabular}[t]{l}$1$\\\end{tabular}}}}%
    \put(-0.00017485,0.19233694){\color[rgb]{0,0,0}\makebox(0,0)[lt]{\lineheight{1.25}\smash{\begin{tabular}[t]{l}$i$\\\end{tabular}}}}%
    \put(-0.00033106,0.01558179){\color[rgb]{0,0,0}\makebox(0,0)[lt]{\lineheight{1.25}\smash{\begin{tabular}[t]{l}$n$\\\end{tabular}}}}%
    \put(-0.00055985,0.08970731){\color[rgb]{0,0,0}\makebox(0,0)[lt]{\lineheight{1.25}\smash{\begin{tabular}[t]{l}$j$\end{tabular}}}}%
    \put(0.20597729,0.15035586){\color[rgb]{0,0,0}\makebox(0,0)[lt]{\lineheight{1.25}\smash{\begin{tabular}[t]{l}$(p,0)$-cable\\\end{tabular}}}}%
    \put(0,0){\includegraphics[width=\unitlength,page=5]{cabling_BKL-gen_7.pdf}}%
    \put(0.00017033,0.31054627){\color[rgb]{0,0,0}\makebox(0,0)[lt]{\lineheight{1.25}\smash{\begin{tabular}[t]{l}Corresponding fence diagrams:\end{tabular}}}}%
    \put(0,0){\includegraphics[width=\unitlength,page=6]{cabling_BKL-gen_7.pdf}}%
    \put(0.96060824,0.64768177){\color[rgb]{0,0,0}\makebox(0,0)[lt]{\lineheight{1.25}\smash{\begin{tabular}[t]{l}$1$\\\end{tabular}}}}%
    \put(0.95855896,0.57653735){\color[rgb]{0,0,0}\makebox(0,0)[lt]{\lineheight{1.25}\smash{\begin{tabular}[t]{l}$p(i-1)$\\\end{tabular}}}}%
    \put(0.96060824,0.36700347){\color[rgb]{0,0,0}\makebox(0,0)[lt]{\lineheight{1.25}\smash{\begin{tabular}[t]{l}$pn$\\\end{tabular}}}}%
    \put(0.96060824,0.44802034){\color[rgb]{0,0,0}\makebox(0,0)[lt]{\lineheight{1.25}\smash{\begin{tabular}[t]{l}$pj$\end{tabular}}}}%
    \put(0.96060824,0.63078442){\color[rgb]{0,0,0}\makebox(0,0)[lt]{\lineheight{1.25}\smash{\begin{tabular}[t]{l}$p$\\\end{tabular}}}}%
    \put(0.96060824,0.55110189){\color[rgb]{0,0,0}\makebox(0,0)[lt]{\lineheight{1.25}\smash{\begin{tabular}[t]{l}$pi$\\\end{tabular}}}}%
    \put(0.96060824,0.52433246){\color[rgb]{0,0,0}\makebox(0,0)[lt]{\lineheight{1.25}\smash{\begin{tabular}[t]{l}$p(i+1)$\\\end{tabular}}}}%
    \put(0.96060824,0.42164455){\color[rgb]{0,0,0}\makebox(0,0)[lt]{\lineheight{1.25}\smash{\begin{tabular}[t]{l}$p(j+1)$\\\end{tabular}}}}%
    \put(0.59548345,0.28255721){\color[rgb]{0,0,0}\makebox(0,0)[lt]{\lineheight{1.25}\smash{\begin{tabular}[t]{l}$1$\\\end{tabular}}}}%
    \put(0.59343416,0.2114128){\color[rgb]{0,0,0}\makebox(0,0)[lt]{\lineheight{1.25}\smash{\begin{tabular}[t]{l}$p(i-1)$\\\end{tabular}}}}%
    \put(0.59548345,0.00188157){\color[rgb]{0,0,0}\makebox(0,0)[lt]{\lineheight{1.25}\smash{\begin{tabular}[t]{l}$pn$\\\end{tabular}}}}%
    \put(0.59548345,0.08289578){\color[rgb]{0,0,0}\makebox(0,0)[lt]{\lineheight{1.25}\smash{\begin{tabular}[t]{l}$pj$\end{tabular}}}}%
    \put(0.59548345,0.26565988){\color[rgb]{0,0,0}\makebox(0,0)[lt]{\lineheight{1.25}\smash{\begin{tabular}[t]{l}$p$\\\end{tabular}}}}%
    \put(0.59548345,0.1859773){\color[rgb]{0,0,0}\makebox(0,0)[lt]{\lineheight{1.25}\smash{\begin{tabular}[t]{l}$pi$\\\end{tabular}}}}%
    \put(0.59548345,0.1592079){\color[rgb]{0,0,0}\makebox(0,0)[lt]{\lineheight{1.25}\smash{\begin{tabular}[t]{l}$p(i+1)$\\\end{tabular}}}}%
    \put(0.59548345,0.05652247){\color[rgb]{0,0,0}\makebox(0,0)[lt]{\lineheight{1.25}\smash{\begin{tabular}[t]{l}$p(j+1)$\\\end{tabular}}}}%
    \put(0,0){\includegraphics[width=\unitlength,page=7]{cabling_BKL-gen_7.pdf}}%
  \end{picture}%
\endgroup%

%% file: fence_diagram_new4.pdf_tex
\begingroup%
  \makeatletter%
  \providecommand\color[2][]{%
    \errmessage{(Inkscape) Color is used for the text in Inkscape, but the package 'color.sty' is not loaded}%
    \renewcommand\color[2][]{}%
  }%
  \providecommand\transparent[1]{%
    \errmessage{(Inkscape) Transparency is used (non-zero) for the text in Inkscape, but the package 'transparent.sty' is not loaded}%
    \renewcommand\transparent[1]{}%
  }%
  \providecommand\rotatebox[2]{#2}%
  \newcommand*\fsize{\dimexpr\f@size pt\relax}%
  \newcommand*\lineheight[1]{\fontsize{\fsize}{#1\fsize}\selectfont}%
  \ifx\svgwidth\undefined%
    \setlength{\unitlength}{382.25674366bp}%
    \ifx\svgscale\undefined%
      \relax%
    \else%
      \setlength{\unitlength}{\unitlength * \real{\svgscale}}%
    \fi%
  \else%
    \setlength{\unitlength}{\svgwidth}%
  \fi%
  \global\let\svgwidth\undefined%
  \global\let\svgscale\undefined%
  \makeatother%
  \begin{picture}(1,0.77716121)%
    \lineheight{1}%
    \setlength\tabcolsep{0pt}%
    \put(0,0){\includegraphics[width=\unitlength,page=1]{fence_diagram_new4.pdf}}%
    \put(-0.000837,0.67218188){\color[rgb]{0,0,0}\makebox(0,0)[lt]{\lineheight{1.25}\smash{\begin{tabular}[t]{l}$1$\\\end{tabular}}}}%
    \put(-0.00055421,0.52520286){\color[rgb]{0,0,0}\makebox(0,0)[lt]{\lineheight{1.25}\smash{\begin{tabular}[t]{l}$i$\\\end{tabular}}}}%
    \put(-0.00087148,0.16622044){\color[rgb]{0,0,0}\makebox(0,0)[lt]{\lineheight{1.25}\smash{\begin{tabular}[t]{l}$n$\\\end{tabular}}}}%
    \put(-0.00133612,0.31676633){\color[rgb]{0,0,0}\makebox(0,0)[lt]{\lineheight{1.25}\smash{\begin{tabular}[t]{l}$j$\end{tabular}}}}%
    \put(0.40255965,0.43994102){\color[rgb]{0,0,0}\makebox(0,0)[lt]{\lineheight{1.25}\smash{\begin{tabular}[t]{l}$(p,0)$-cable\\\end{tabular}}}}%
    \put(0,0){\includegraphics[width=\unitlength,page=2]{fence_diagram_new4.pdf}}%
    \put(0.91999696,0.76254383){\color[rgb]{0,0,0}\makebox(0,0)[lt]{\lineheight{1.25}\smash{\begin{tabular}[t]{l}$1$\\\end{tabular}}}}%
    \put(0.91999706,0.5833683){\color[rgb]{0,0,0}\makebox(0,0)[lt]{\lineheight{1.25}\smash{\begin{tabular}[t]{l}$p(i-1)$\\\end{tabular}}}}%
    \put(0.91999696,0.00382132){\color[rgb]{0,0,0}\makebox(0,0)[lt]{\lineheight{1.25}\smash{\begin{tabular}[t]{l}$pn$\\\end{tabular}}}}%
    \put(0.91999696,0.2044294){\color[rgb]{0,0,0}\makebox(0,0)[lt]{\lineheight{1.25}\smash{\begin{tabular}[t]{l}$pj$\end{tabular}}}}%
    \put(0.91999696,0.71001684){\color[rgb]{0,0,0}\makebox(0,0)[lt]{\lineheight{1.25}\smash{\begin{tabular}[t]{l}$p$\\\end{tabular}}}}%
    \put(0.91999696,0.49147617){\color[rgb]{0,0,0}\makebox(0,0)[lt]{\lineheight{1.25}\smash{\begin{tabular}[t]{l}$pi$\\\end{tabular}}}}%
    \put(0.91999696,0.41907286){\color[rgb]{0,0,0}\makebox(0,0)[lt]{\lineheight{1.25}\smash{\begin{tabular}[t]{l}$p(i+1)$\\\end{tabular}}}}%
    \put(0.91999696,0.13074942){\color[rgb]{0,0,0}\makebox(0,0)[lt]{\lineheight{1.25}\smash{\begin{tabular}[t]{l}$p(j+1)$\\\end{tabular}}}}%
    \put(0,0){\includegraphics[width=\unitlength,page=3]{fence_diagram_new4.pdf}}%
    \put(0.63307459,0.73384997){\color[rgb]{0,0,0}\makebox(0,0)[lt]{\lineheight{1.25}\smash{\begin{tabular}[t]{l}$p$\\\end{tabular}}}}%
    \put(0,0){\includegraphics[width=\unitlength,page=4]{fence_diagram_new4.pdf}}%
    \put(0.63307459,0.6079462){\color[rgb]{0,0,0}\makebox(0,0)[lt]{\lineheight{1.25}\smash{\begin{tabular}[t]{l}$p$\\\end{tabular}}}}%
    \put(0,0){\includegraphics[width=\unitlength,page=5]{fence_diagram_new4.pdf}}%
    \put(0.63307459,0.44501165){\color[rgb]{0,0,0}\makebox(0,0)[lt]{\lineheight{1.25}\smash{\begin{tabular}[t]{l}$p$\\\end{tabular}}}}%
    \put(0,0){\includegraphics[width=\unitlength,page=6]{fence_diagram_new4.pdf}}%
    \put(0.63307459,0.3191771){\color[rgb]{0,0,0}\makebox(0,0)[lt]{\lineheight{1.25}\smash{\begin{tabular}[t]{l}$p$\\\end{tabular}}}}%
    \put(0,0){\includegraphics[width=\unitlength,page=7]{fence_diagram_new4.pdf}}%
    \put(0.63307459,0.15557514){\color[rgb]{0,0,0}\makebox(0,0)[lt]{\lineheight{1.25}\smash{\begin{tabular}[t]{l}$p$\\\end{tabular}}}}%
    \put(0,0){\includegraphics[width=\unitlength,page=8]{fence_diagram_new4.pdf}}%
    \put(0.63307459,0.02830197){\color[rgb]{0,0,0}\makebox(0,0)[lt]{\lineheight{1.25}\smash{\begin{tabular}[t]{l}$p$\\\end{tabular}}}}%
    \put(0,0){\includegraphics[width=\unitlength,page=9]{fence_diagram_new4.pdf}}%
    \put(0.63307214,0.52941365){\color[rgb]{0,0,0}\makebox(0,0)[lt]{\lineheight{1.25}\smash{\begin{tabular}[t]{l}$p$\\\end{tabular}}}}%
    \put(0,0){\includegraphics[width=\unitlength,page=10]{fence_diagram_new4.pdf}}%
    \put(0.63307214,0.23881953){\color[rgb]{0,0,0}\makebox(0,0)[lt]{\lineheight{1.25}\smash{\begin{tabular}[t]{l}$p$\\\end{tabular}}}}%
  \end{picture}%
\endgroup%

%% file: delta_cabling_v2.pdf_tex
\begingroup%
  \makeatletter%
  \providecommand\color[2][]{%
    \errmessage{(Inkscape) Color is used for the text in Inkscape, but the package 'color.sty' is not loaded}%
    \renewcommand\color[2][]{}%
  }%
  \providecommand\transparent[1]{%
    \errmessage{(Inkscape) Transparency is used (non-zero) for the text in Inkscape, but the package 'transparent.sty' is not loaded}%
    \renewcommand\transparent[1]{}%
  }%
  \providecommand\rotatebox[2]{#2}%
  \newcommand*\fsize{\dimexpr\f@size pt\relax}%
  \newcommand*\lineheight[1]{\fontsize{\fsize}{#1\fsize}\selectfont}%
  \ifx\svgwidth\undefined%
    \setlength{\unitlength}{623.6220905bp}%
    \ifx\svgscale\undefined%
      \relax%
    \else%
      \setlength{\unitlength}{\unitlength * \real{\svgscale}}%
    \fi%
  \else%
    \setlength{\unitlength}{\svgwidth}%
  \fi%
  \global\let\svgwidth\undefined%
  \global\let\svgscale\undefined%
  \makeatother%
  \begin{picture}(1,0.59699363)%
    \lineheight{1}%
    \setlength\tabcolsep{0pt}%
    \put(0.51963552,0.58803373){\color[rgb]{0,0,0}\makebox(0,0)[lt]{\lineheight{1.25}\smash{\begin{tabular}[t]{l}$1$\\\end{tabular}}}}%
    \put(0.51554463,0.33863238){\color[rgb]{0,0,0}\makebox(0,0)[lt]{\lineheight{1.25}\smash{\begin{tabular}[t]{l}$pn$\\\end{tabular}}}}%
    \put(0,0){\includegraphics[width=\unitlength,page=1]{delta_cabling_v2.pdf}}%
    \put(0.14236282,0.49965832){\color[rgb]{0,0,0}\makebox(0,0)[lt]{\lineheight{1.25}\smash{\begin{tabular}[t]{l}$1$\\\end{tabular}}}}%
    \put(0.14099918,0.43081538){\color[rgb]{0,0,0}\makebox(0,0)[lt]{\lineheight{1.25}\smash{\begin{tabular}[t]{l}$n$\\\end{tabular}}}}%
    \put(0.51963552,0.54292873){\color[rgb]{0,0,0}\makebox(0,0)[lt]{\lineheight{1.25}\smash{\begin{tabular}[t]{l}$p$\\\end{tabular}}}}%
    \put(0,0){\includegraphics[width=\unitlength,page=2]{delta_cabling_v2.pdf}}%
    \put(0.10372647,0.4652379){\color[rgb]{0,0,0}\makebox(0,0)[lt]{\lineheight{1.25}\smash{\begin{tabular}[t]{l}$\delta_n$\\\end{tabular}}}}%
    \put(0.31281735,0.48620144){\color[rgb]{0,0,0}\makebox(0,0)[lt]{\lineheight{1.25}\smash{\begin{tabular}[t]{l}$(p,0)$-cable\\\end{tabular}}}}%
    \put(0,0){\includegraphics[width=\unitlength,page=3]{delta_cabling_v2.pdf}}%
    \put(0.46863438,0.14136652){\color[rgb]{0,0,0}\makebox(0,0)[lt]{\lineheight{1.25}\smash{\begin{tabular}[t]{l}isotopy\end{tabular}}}}%
    \put(0,0){\includegraphics[width=\unitlength,page=4]{delta_cabling_v2.pdf}}%
    \put(0.4644597,0.28223001){\color[rgb]{0,0,0}\rotatebox{32.30833198}{\makebox(0,0)[lt]{\lineheight{1.25}\smash{\begin{tabular}[t]{l}isotopy\end{tabular}}}}}%
    \put(0,0){\includegraphics[width=\unitlength,page=5]{delta_cabling_v2.pdf}}%
  \end{picture}%
\endgroup%